\documentclass[a4paper, 11 pt, twoside]{amsart}
\usepackage{srollens-en}[2011/03/03]
\usepackage[left = 3.5cm, right = 3.5cm, headsep = 6mm,
footskip = 10mm, top = 35mm, bottom = 35mm, footnotesep=5mm, headheight =
2cm]{geometry}

\newlist{prooflist}{description}{1}
\setlist[prooflist]{font=\normalfont \itshape, labelindent = \parindent, leftmargin = 0pt}

\usepackage{tikz}
\usepackage{tikz-cd}
\usetikzlibrary{patterns,decorations.pathreplacing}

\makeatletter
\g@addto@macro\@floatboxreset\centering
\makeatother

\usepackage[unicode,bookmarks, pdftex]{hyperref}
\hypersetup{colorlinks=true,citecolor=NavyBlue,linkcolor=NavyBlue,urlcolor=Orange, pdfpagemode=UseNone, breaklinks=true}

\title{I-surfaces from surfaces with one exceptional unimodal point}

    \tikzset{
curve/.style = { every node/.style = { font = \scriptsize}, thick},
E1/.style = { MidnightBlue, curve},
Ei/.style = { NavyBlue, curve},
branch/.style = { red, curve,every node/.style = { font = \scriptsize}},
fibre/.style = {Gray,curve},
F/.style = {Orange, curve , thin},
2F/.style = {White!70!Orange,   curve, very thick},
GG/.style = {PineGreen,curve},
G/.style = {PineGreen,curve,  thin},
2G/.style = {White!70!PineGreen,  curve, very thick},
pfeil/.style = {->, every node/.style = { font = \scriptsize}},
sing/.style = {->, every node/.style = { font = \scriptsize}},
}

\author{S\"onke Rollenske}
\address{S\"onke Rollenske\\FB 12/Mathematik und Informatik\\
Philipps-Universit\"at Marburg\\
Hans-Meerwein-Str. 6\\
35032 Marburg\\
Germany}
\email{rollenske@mathematik.uni-marburg.de}

\author{Diana Torres}
\address{Diana Torres\\FB 12/Mathematik und Informatik\\
Philipps-Universit\"at Marburg\\
Hans-Meer\-wein-Str. 6\\
35032 Marburg\\
Germany}
\email{dicato98@gmail.com}

\newlength{\Breite}
\begin{document}
\begin{abstract}
We complement recent work of Gallardo, Pearlstein, Schaffler, and Zhang, showing that
the stable surfaces with $K_X^2 =1$ and $\chi(\ko_X) = 3$ they construct are indeed the only ones arising from imposing an exceptional unimodal double point.

In addition, we explicitly describe the birational type of the surfaces constructed from singularities of type $E_{12}$, $E_{13}$, $E_{14}$.
\end{abstract}

\maketitle
\setcounter{tocdepth}{1}
\tableofcontents

\section{Introduction}

This paper is inspired by a recent work of  Gallardo, Pearlstein,  Schaffler, and  Zhang \cite{GPSZ22}, so let us briefly set out the context. Classically, the coarse moduli space $\gothM_{1,3}$ of canonical models of surfaces of general type with $K_X^2=1$ and $\chi(\ko_X) = \chi(X) = 3$ is an irreducible and unirational variety of dimension $28$. Sometimes these run under the name of special Horikawa surfaces, sometimes they are called classical I-surfaces for brevity and they can all be described as hypersurfaces of degree $10$ in the weighted projective space $\IP(1,1,2,5)$. 

Nowadays, the Gieseker moduli  space $\gothM_{1,3}$ is known to admit a modular compactification $\overline{\gothM}_{1,3}$, the moduli space of stable surfaces, sometimes called KSBA-moduli space after Koll\'ar, Shepherd-Barron, and Alexeev (compare \cite{KollarModuli}). Stable I-surfaces have been used as a testing ground for various approaches to understand the compactification and the surfaces parametrised by it. The approach of the first named author and his collaborators \cite{FPR17, FPRR22, CFPRR} involved trying to slowly increase the generality of which singularities we allow, partly inspired by Hodge theoretic aspects \cite{GGR21, CFPR}.

Gallardo,Pearlstein,  Schaffler, and  Zhang used a different approach in \cite{GPSZ22}: instead of working inside the realm of stable surfaces they consider degenerations to non-log-canonical surfaces and consider the stable replacement of the limit surface, whose existence is guaranteed by the properness of the stable compactification. 

More precisely, they consider  hypersurfaces in $\IP(1,1,2,5)$ with a unique singular point which is one of the exceptional unimodal double points from Arnold's classification \cite{Arnold76}, given by an explicit type of equation. From these, they describe the stable replacement via an explicit weighted blow up. Counting parameters they find eight new divisors in the closure of the classical component. 
These divisors parametrise reducible surfaces, where one component is  a so-called K3-tail, the two-dimensional equivalent of an elliptic tail on a stable curve. 
Because their approach starts with explicit equations, they cannot exclude that other degenerations with a unique singular point of the given type exist \cite[Rem. 4.5]{GPSZ22}.
In addition, they did not provide an explicit geometric description of the second component in case of the singularities $E_{12}$, $E_{13}$, $E_{14}$.

We use a different approach to the same surfaces and are able to clarify both points.
 \begin{custom}[Theorem]
 Let $W$ be a Gorenstein surface with $K_W^2 = 1$, $\chi(W) = 3$ and ample canonical bundle. 
 \begin{enumerate}
  \item If $W$ has a unique singular point $w$ which is an exceptional unimodal double point, then $W$ is in the closure of one of the families constructed in \cite{GPSZ22}.
  \item If in addition the point $w$ is of type $E_{12}$, $E_{13}$, or $E_{14}$, then the minimal resolution of $W$ is a minimal elliptic surface with a bisection as explicitly described in Section \ref{sect: En}.
 \end{enumerate}
\end{custom}
The paper is organised as follows: after some general remarks in Section \ref{sect: prelims}, we treat the $E_n$ singularities in Section \ref{sect: En} and the other cases in Section \ref{sect: Z and W}. In both cases there will be a last step where we have to analyse every individual case. 

For the convenience of the reader, we add a final Section \ref{sect: application to moduli} where we spell out, what our results mean for the study of the compactification of the moduli space $\overline\gothM_{1,3}$.

\subsection*{Acknowledgements}
We would like to thank Luca Schaffler for several discussions on the topics of this article. S.R.\ is grateful for support by the DFG. D.T. is supported by a bridging position of Marburg University Research Academy (MARA). 

We thank the two anonymous referees for their many thoughtful comments that hopefully lead to an improvement of the presentation.

\section{Gorenstein surfaces with one  elliptic Gorenstein singularity}
\label{sect: prelims}

Recall that a normal Gorenstein surface singularity $(W,w)$ is called minimally elliptic  by Laufer \cite{Laufer77}, or just elliptic by Reid \cite[4.12]{Reid-chapters} if for some resolution $\varphi\colon X\to W$ we have $R^1\varphi_*\ko_X\isom \ko_w= \IC_w$. From the point of view of singularity theory, this class is quite well behaved but includes many important examples.
\begin{exam} 
Let $(W,w)$ be a Gorenstein  log canonical surface singularity, the classification of which can for example be found in \cite[Sect. 3.3]{KollarSMMP}.

If $(W,w)$ is canonical, it is one of the classical ADE singularities, thus rational and not elliptic. Of these we will only use the $A_n$ singularities, which are locally isomorphic to $\left(\{x^2+y^2 + z^{n+1}=0\}, 0\right)$. The exceptional divisor in the minimal resolution is a string of $(-2)$-curves.

If $(W,w)$ is strictly log-canonical, then it is  a simple elliptic singularity (contraction of a smooth elliptic curve) or a cusp singularity (contraction of a cycle of rational curves). Both are elliptic in the above sense.
\end{exam}
\begin{exam}\label{ex: unimodal}
The most important measures of the complexity of an isolated hypersurface singularity $Z$ are its Milnor number $\mu(W, w)$  and its modality $m(W, w)$.
The latter can be characterised as the dimension of the $\mu$-constant stratum in a versal deformation minus $1$ \cite{MR0440066}. 

The singularities of modality zero are exactly the ADE singularities. Arnold classified unimodal $(m(W, w)=1)$ and bimodal $(m(W, w) = 2)$ hypersurface singularities in \cite{Arnold76} and it was subsequently realised that these are elliptic \cite{Kulikov75, Karras80}. 
 Some information about the types we use are given in Table \ref{tab: info En} and Table \ref{tab: info Z and W}. 
 
 These are in fact special cases of elliptic Kodaira singularities \cite{Ebeling-Wall}, which are constructed as follows: start with a singular fibre $\bar E$ of a relatively minimal elliptic fibration as classified by Kodaira \cite{Kodaira63}. Then blow up $\bar E$ repeatedly in smooth points (on reduced components) and contract the strict transform of $\bar E$. 

 Conversely, the contraction of a configuration of curves as given in the tables will result in a singularity of the given type, see \cite[Sect. 3]{Ebeling-Wall} or \cite{wall83}. 
 \end{exam}

We now consider the following situation: let $W$ be a complex projective  surface with exactly one singular point $w$, which is an elliptic Gorenstein singularity, and such that $K_W$ is ample. 
Let us consider the diagram
\begin{figure}[h]
    \begin{equation}\label{Diagram1}
\begin{tikzcd}
 & X \arrow{dr}{f}\arrow{dl}[swap]{\phi}\\
  W && S
\end{tikzcd}
\end{equation}
\end{figure}
where $X$ is the minimal resolution of $W$ and $S$ is a minimal model for $X$. 
We can write 
\begin{equation}\label{eq: pullback canonical} 
 \phi^*K_W = K_X + E
\end{equation}
where $E = \sum_i a_i E_i$ is the fundamental cycle (all equally fundamental cycles agree in this case, compare \cite[4.21]{Reid-chapters}).  In this paper, we assume that $w$ is an unimodal singularity, and therefore, we know that $E$ is reduced, see the references in Example \ref{ex: unimodal}.

\begin{lem}\label{lem: basic stuff}
 In the above situation we  have the following:
 \begin{enumerate}
  \item 
$ \chi(X) = \chi(W)-1$.
\item 
$ p_g(X) = \begin{cases}
             p_g(W)-1 & \text{ if $w$ is not a base-point of $|K_W|$;}\\
            p_g(W) & \text{ if $w$ is  a base-point of $|K_W|$.}\\
           \end{cases}
$

  \item For the plurigenera we have  $P_m(X)\leq P_m(W)$ for every $m>0$.
\item If $S$ is of Kodaira dimension $\kappa(S)\geq 0$ then $K_X^2\leq K_S^2\leq K_W^2$.
  \end{enumerate}
\end{lem}
\begin{proof}
 The first statement follows from the definition, the Leray spectral sequence and 
 \[\chi(\ko_X) = \chi(\phi_*\ko_X) -\chi(R^1\phi_*\ko_X) = \chi(\ko_W)-\chi(\IC_w).\]
 
 Note that $E$ has arithmetic genus one, so by adjunction we have the exact sequence
 \[ 0\to \omega_X\to \phi^*\omega_W = \omega_X(E) \to \ko_E\to 0.\]
 Taking cohomology and using the projection formula we get
 \[ 0 \to H^0(K_X) \to H^0(K_W) \to H^0(\IC_w) \to \dots\]
 and we get the two possibilities given for the geometric genus. 

Item  $(iii)$  follows from the analogous inclusion $H^0(mK_X) \into H^0(\phi^*mK_W) = H^0( mK_W)$.

For the last statement, note that $K_X^2\leq K_S^2$ and $P_m(X) = P_m(S)$ because $S$ is a minimal model for $X$. Since by assumption both $K_S$ and $K_W$ are nef, we can compare the leading terms of the formula for the plurigenera via the  asymptotic Riemann-Roch for a nef divisor \cite[Prop. 1.31]{Debarre} getting
\[
 \frac12 m^2 K_S^2 + \text{l.o.t.}  = P_m(S) \leq P_m(W) = \frac 12 m^2 K_W^2 +\text{ l.o.t.} 
\]
This shows the missing inequality $K_S^2\leq K_W^2$.
\end{proof}

It was observed in \cite{FPR17} that methods involving the canonical ring often extend to Gorenstein surfaces. For the invariants of particular interest to us, the result reads as follows.
\begin{prop}\label{prop: min res p_g=2}
Let $W$ be a normal projective surface with $K_W^2 = 1$, $K_W$ ample and Cartier and $p_g(W) =2$. Then $W$ is a double cover of $\IP(1,1,2)$ via the bicanonical map branched over the singular point and a branch divisor $\Delta$ disjoint from the singular point. 

If $W$ has a unique singularity which is an elliptic point, then the minimal resolution $X$ satisfies $p_g(X) = 1$ and $q(X)=0$ and  the effective canonical divisor $K_X$  is connected and the sum of a reduced and irreducible curve not contracted by $\phi$ and possibly an effective combination of $\phi$-exceptional curves.
\end{prop}
\begin{proof}
 Consider a general canonical curve $C\in |K_W|$. Then with the same proof as in \cite[Lemma 4.1]{FPR15b} $C$ is an irreducible and reduced Gorenstein curve. The arguments from \cite[Thm 3.3]{FPR17} apply verbatim in our context but let us sketch them for the benefit of the reader: the restriction $L = K_W|_C$ is a line bundle on $C$ with $h^0(L) = 1$ and $2L = K_C$. Its section ring $R(C, L)$ is easily calculated. Reid's hyperplane section principle gives us the structure of the canonical ring of $W$, which then realises $W$ as  a hypersurface of degree ten contained in the smooth locus of $\IP(1,1,2,5)$. Thus the bicanonical map 
 factors as \[ 
  \begin{tikzcd}
W \rar &  \IP(1,1,2) \rar[hookrightarrow]{|\ko(2)|} & \IP^3
  \end{tikzcd}
 \]
and realises $W$ as a double cover of the quadric cone branched over the vertex and a quintic section not containing the vertex. 
 In particular,  the base-point of the canonical linear system is a smooth point of $W$, so  the unimodal point $w$ is not a base-point of $K_W$ and by Lemma \ref{lem: basic stuff} we have $p_g(X) = 1$. 
 
 Since every canonical curve on $W$ is reduced and irreducible,  $ K_X = \varphi^*K_W - E$ contains exactly one reduced and non-exceptional component. 
 \end{proof}
\section{Surfaces with exceptional unimodal double points of type $E_n$}\label{sect: En}
In this section we consider the situation of \eqref{Diagram1} where $K_W^2 = 1$, $p_g(W) = 2$ and $q(W)=0$ and $W$ has  a unique singular point of type $E_{12}, E_{13}$, or  $E_{14}$ in Arnold's notation. In these cases the exceptional divisor $E$ of $\varphi$ is reduced and we have
\[ E^2 = - 1,\; K_X.E = 1, \text{ and } K_X^2 = 0.\]
Refer to Table \ref{tab: info En} for information on the singularities of Type $E_{12}, E_{13}$, or  $E_{14}$, extracted from the references given in Example \ref{ex: unimodal}
\begin{table}
  \caption{The exceptional unimodal points of type $E_n$.}
  \label{tab: info En}
 \begin{tabular}{ccccc}
 \toprule
 Type & Kodaira fibre & Blow ups & $E$ & equation \\
 \midrule
$E_{12}$&
 $II$&$(1)$ &
 \begin{minipage}[c]{\Breite}
 \centering
  \begin{tikzpicture}[curve]
\begin{scope}[scale=0.5]
\draw (1,2) to[out  = -90, in = 0] (0,0) to [in = 90,  out = 0] (1, -2);
\node at (1,-2.5) {$E_1$};
\node at (0.9,2.5){$-1$};

\end{scope} 
\end{tikzpicture} 
 \end{minipage}
&$z^3+y^7+ay^5z$   
\\
$E_{13}$& $III$&$(1,0)$& 
\begin{minipage}[c]{\Breite}
 \centering
\begin{tikzpicture}[curve]
\begin{scope}[ scale=0.5, xshift = 7
cm]
\draw (1,2) to[out  = -135, in = 90] (0,0) to [in = 135,  out = -90] (1, -2);
\draw (-1,2) to[out  = -45, in = 90] (0,0) to [in = 45,  out = -90] (-1, -2);
\node at (-1,-2.5) {$E_1$};
\node at (-1.2,2.5){$-3$};
\node at (1,-2.5) {$E_2$};
\node at (1,2.5){$-2$};

\end{scope}
\end{tikzpicture}
\end{minipage}
&
$z^3+y^5z +ay^8$ 
\\
  $E_{14}$&$IV$&$(1,0,0)$ & 
\begin{minipage}[c]{\Breite}
 \centering
 \begin{tikzpicture}[curve]
\begin{scope}[scale=0.5, xshift = 14
cm]
\draw (0:2) to (180:2);
\draw (60:2.3) to (240:2.3);
\draw (120:2.3) to (300:2.3);
\node at (-1.2,-2.5) {$E_2$};
\node at (1.2,-2.5) {$E_3$};
\node at (1,2.5){$-2$};
\node at (-1.2,2.5){$-2$};
\node at (-2.6,0) {$E_1$};
\node at (2.6,0) {$-3$};
\end{scope}
  \end{tikzpicture}
\end{minipage}
  &  $z^3+ y^8 +ay^6z$   \\
\bottomrule
\end{tabular}
 \end{table}

Let us start with three  constructions, which all follow the same pattern explained in Figure \ref{fig: construction diagram}.
\begin{figure}[h]
 \caption{Schematic construction of surfaces with $E_n$ type singularities}
 \label{fig: construction diagram}
 
 \begin{tikzcd}
  \bar X \dar[swap]{\bar\theta} & \hat X \lar[swap]{\bar \sigma} \arrow{dr}{\sigma} \arrow{dl}[swap]{\hat\theta}\\
  P \dar && X  \arrow[dashed]{ll}[swap]{\theta}\rar {\phi} \dar{\pi} & W\\
   \IP^1 \arrow[equal]{rr} &&\IP^1
   \end{tikzcd}
 \begin{minipage}{.4\textwidth}
  \begin{itemize}
   \item P a Hirzebruch surface
   \item $\bar \theta$ a double cover
   \item $\bar \sigma$ the minimal resolution
   \item $\sigma$ a blow up
   \item $\phi$ the resolution of an $E_n$ singularity
  \end{itemize}
 \end{minipage}
\end{figure}
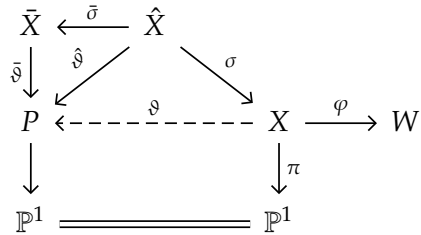


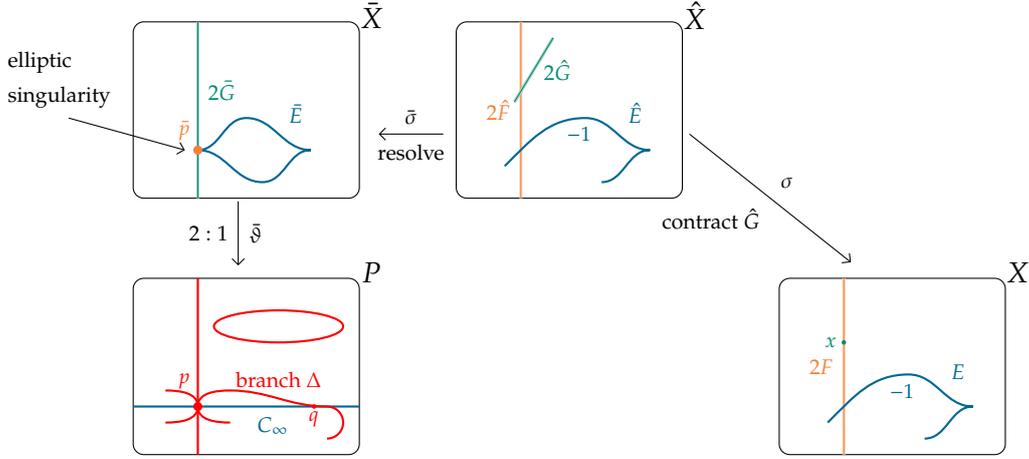
\begin{figure}
   \caption{Schematic construction, type $E_{12}$ (picture)} \label{fig: picture E12}
  \begin{tikzpicture}[scale = .85]
   
   \begin{scope}
        \draw[rounded corners] (-2.0, 1.5) rectangle (1.5, -1.25);
            \node[] at (1.7,1.6) {$P$};
        \draw[E1] (-2, -.5) to node[below right] {$C_\infty$} ++(3.5, 0);

\draw[branch] (-1, -1.25) to ++ (0, 2.75);
\draw [branch] (-1.5, -.25) to[out= 0, in =90] ++(.5, -.25) to[in = 0,out = -90] ++(-.5, -.25);
\draw [branch] (-.5, -.75) to[out= 180, in =-90] ++(-.5, .25) to[in = 180,out = 90] ++(.5, .25)
to[out = 0 , in  = 180] 
node[above] {branch $\Delta$}
++ (1.5, -.25) to[out= 0, in =90] ++(.25, -.25) to[in = 0,out = -90] ++(-.25, -.25);
\node[branch, scale=0.7] at (-1.2,-.1) {$p$};
\fill[branch] (-1, -.5) circle (2pt);
\draw[branch] (.25,.75) ellipse (1cm and .25cm);

\fill[branch] (0.8, -.5) circle (1pt);
\node[branch, scale=0.7] at (0.8,-.7) {$q$};
  \end{scope}

    \begin{scope}[xshift = 0cm, yshift = 4cm] 
   
        \draw[rounded corners] (-2.0, 1.5) rectangle (1.5, -1.25);
            \node[] at (1.7,1.6) {$\bar X$};
  
\draw[2G] (-1, -1.25) to  ++ (0, 2.75);
\draw[G] (-1, -1.25) to node[above right] {$2\bar G$} ++ (0, 2.75);

        \draw[E1] (-1, -.5)  to[out = 0, in = 180] ++(.75, .5) to[out = 0, in = 180] node[above right] {$\bar E$} ++(1,- .5) to[out = 180, in = 0]  ++(-.75,- .5)
         to[out = 180, in = 0] ++(-1,.5);
\draw[pfeil] (-0.375, -1.3) -> node[right] {$\bar\theta$} node[left]{$2:1$} ++ (0, -1); 

 \fill[F] (-1, -.5) circle (2pt);
 \node[F, scale=0.7] at (-1.2, -.2) {$\bar{p}$};
 \draw[<-] (-1.2, -.5)  -- ++(-1.8, .5) node[above, text width=1.6cm] {\scriptsize elliptic singularity};
  \end{scope}

      \begin{scope}[xshift = 5cm, yshift = 4cm]
   
        \draw[rounded corners] (-2.0, 1.5) rectangle (1.5, -1.25);
            \node[] at (1.7,1.6) {$\hat X$};
  
\draw[2F] (-1, -1.25) to ++ (0, 2.75);
\draw[F] (-1, -1.25) to node[left] {$2\hat F$} ++ (0, 2.75);
\draw[2G] (-1.1, .25) to ++ (.6, 1);
\draw[G] (-1.1, .25) to  node[right] {$2\hat G$}  ++ (.6, 1);

        \draw[E1] (-1, -.5) ++ (-.25, -.25) to[out = 45, in = 180]  ++(1.25, .75) to[out = 0, in = 180]node[above right] {$\hat E$}  ++(1,- .5) to[out = 180, in = 0] ++(-.75,- .5) 
         ;
         \node[E1,scale=0.7] at (-.1,-.25){$-1$}; 

         \draw[pfeil] (-2.2, -0.25) -> node[below] {resolve} node[above]{$\bar\sigma$} ++ (-1, 0); 

         \draw[pfeil] (1.61, -0.25) -> node[below left] {contract $\hat G$} node[above right]{$\sigma$} ++ (2.5, -2); 
         
  \end{scope}
  
      \begin{scope}[xshift = 10cm, yshift = 0cm]
   
        \draw[rounded corners] (-2.0, 1.5) rectangle (1.5, -1.25);
            \node[] at (1.7,1.6) {$X$};
  
\draw[2F] (-1, -1.25) to ++ (0, 2.75);
\draw[F] (-1, -1.25) to node[left] {$2F$}++ (0, 2.75);
\fill[G] (-1, .5) circle (1pt);
\node[G, scale=0.7] at (-1.2, .5) {$x$};

        \draw[E1] (-1, -.5) ++ (-.25, -.25) to[out = 45, in = 180]  ++(1.25, .75) to[out = 0, in = 180] node[above right] {$ E$}++(1,- .5) to[out = 180, in = 0] ++(-.75,- .5);
        \node[E1,scale=0.7] at (-.1,-.25){$-1$};
  \end{scope}
\end{tikzpicture}

\end{figure}

On a Hirzebruch surface $P = \IF_n\to \IP^1$ we denote by $\Gamma_p$ the fibre through a point $p$ and  by $\Gamma$ a general fibre. If $n>0$ then $C_\infty$ denotes the unique negative section. In $\IF_0 = \IP^1\times \IP^1$ we fix an arbitrary section $C_\infty$ of self-intersection zero. 

Since the Picard group of a Hirzebruch surface does not have $2$-torsion, a double cover $\bar X \to P$ is uniquely determined by a $2$-divisible divisor $\Delta$ 
and the singularities of $\bar X$ are controlled by the singularities of $\Delta$. One particular type plays an important role.

\begin{rem}\label{rem: admissible33}
Assume that $\Delta$ has a $[3,3]$-point at $p$, a triple point with an infinitesimally near triple point and consider the   resulting surface singularity, which is elliptic of degree one in the sense of Section \ref{sect: prelims}.

We are interested in the case where the exceptional divisor $E$ in a minimal resolution is (locally analytically isomorphic to) the blow up of a Kodaira fibre of type  $I_n$ for $n\geq 0$ in a smooth point. The possible branch divisors can easily be worked out by hand (compare \cite{FPR17, Anthes20}) and the corresponding singularities in Arnold's list are (see \cite{Karras80}  )
\[T_{2,3,n+6}: x^2 + y^3+ z^{n+6} + \lambda xyz = \left(x+ \frac\lambda2 yz\right)^2 + y^3 - \frac{\lambda^2}{4}y^2z^2 + z^{n+6}. \]
For $n = 0$ we get a simple elliptic singularity where the value of $\lambda$ (with  $\lambda^6\neq 432$)  determines the elliptic curve and   for $n\geq 1$ we get a cusp singularity and  can choose  $\lambda=1$.

In suitable local analytic coordinates the branch divisor $\Delta$ is thus given by $y^3 + \lambda^2y^2z^2 + z^{n+6}$.
Blowing up twice at the origin it is easy to check that we can write the germ of $\Delta$ as $\Delta' + \Delta''$, where
\begin{enumerate}
 \item $\Delta''$ is smooth at $p$,
 \item $\Delta'$ has an $A_{n+3}$ singularity at $p$,
 \item $(\Delta'.\Delta'')_p = 4$.
\end{enumerate}

\end{rem}
In the construction of examples we leave some claims on intersection numbers to the reader, because we will have to \emph{reverse engineer} these constructions anyway in Section \ref{sect: proof of En}.

\begin{exam}
\label{ex: E12}
 Let $P = \IF_0  = \IP^1\times \IP^1$. Consider a curve $\Delta = \Gamma_p + \Delta'\in |4C_\infty + 6 \Gamma|$ with the following properties
 \begin{itemize}
  \item $\Delta$ has a $[3,3]$-point at $p = \Gamma_p \cap C_\infty$ as in Remark \ref{rem: admissible33},
  \item $\Delta|_{C_\infty} = 3p + 3q$ for a second point $q$,
  \item $\Delta$ is smooth outside $p$.
\end{itemize}
We claim that such divisors form an irreducible open subset of a linear system.
Let us spell this out concretely: Consider the $\IZ^2$ graded polynomial ring $S = \IC[x_0, x_1, t_0, t_1]$ with degrees
\[ \mat{ x_0 & x_1& t_0 &t_1\\ 1& 1& 0&0\\0&0&1&1}.\]
Let $C_\infty = \{ x_1 = 0 \}$, $\Gamma_p = \{t_0 = 0 \}$, $p = ((1:0), (0:1))$ and $q = ((1:0), (1:0))$.
Then $\Delta = \Gamma_p+ \Delta'$ is defined by an equation $f = t_0f'$ of bidegree $\mat{4\\6}$. Let us make explicit the condition for $\Delta$ to have a $[3,3]$-point: $t_0, x_1$ are local (inhomogeneous) coordinates at $p$ and $f = f_0 f'$  has at least a  triple point at $p$ if and only if $f'$ does not contain the monomials $1, t_0, x_1$. Assuming the triple point we check for the infinitesimally near triple point by blowing up $p$. The relevant coordinates are $T_0, x_1$ where $t_0 = T_0 x_1$ and the strict transform is given by 
\[ \frac {f(T_0x_1, x_1)}{x_0^3} = T_0 \frac{f'(T_0x_1, x_1)}{x_1^2}.\]
This has at least a triple point in $0$ if $f'$ does not contain the monomials $t_0x_1, x_1^2, x_1^3$.

Thus in total, the first three conditions imply that 
\[f' \in (t_0^2,\, t_0x_1^2,\, x_1^4) \cap   (t_0^2t_1^3,\, x_1) 
,\]
therefore the divisors $\Delta$ with at least these singularities form a linear system. If we can exhibit one example with exactly the requested singularities, an open subset of the linear system will have this property. One explicit equation showing the existence of such branch divisors is
\[ t_0\left(x_0^4t_1^3t_0^2+x_1^4t_1^5+x_0^3x_1t_0^5+ x_0x_1^3t_0^5 \right).\]

Now let $ \bar \theta \colon \bar X \to P$ be the double cover branched over $\Delta$. The preimage $2\bar{G} = \bar\theta^*\Gamma_p$ is a double fibre passing through an elliptic singularity of degree one of type $T_{2,3,n+6}$, where  we assume $n = 0$ or $n=1$.

Let $\bar\sigma\colon \hat X \to \bar X$ be the resolution of the elliptic point and call the exceptional divisor $\hat F$ and $\hat G = {(\inverse{\bar\sigma})}_*\bar G$. Then $2\hat F + 2 \hat G$ is a (double) fibre of the fibration $\hat X \to P \to \IP^1$, so 
\[0 = (\hat F + \hat G)^2 =\hat F^2 + 2 \hat F \hat G + \hat G^ 2  = -1 + 2 + \hat G^2  = 1+\hat G^2\] and $\hat G$ is a $(-1)$-curve.
We denote the contraction of $\hat G$ by   $\sigma\colon \hat X \to X$. 
Then the induced map $\pi\colon X \to \IP^1$ is a relatively minimal elliptic fibration.

Now we  follow the  section $C_\infty\subset P$ along this construction, compare again Figure \ref{fig: picture E12}: the pullback under the double cover gives a bisection $\bar E$ of the fibration $\bar X \to \IP^1$, where $\bar E \to C_\infty$ is branched over $3p + 3q$, thus has two cusp singularities $A_2$. Resolving the elliptic point $\hat X \to \bar X$ resolves the singularity over $p$, while the contraction of $\hat G$ happens away from the strict transform. Therefore in $X$ we get  a bisection $E$ which is a rational curve with one cusp singularity $A_2$ and $E^2 = -1$.
The contraction of $E$ gives a surface $W$ with one singularity of type $E_{12}$, compare Example \ref{ex: unimodal}.
It is straightforward to check that $K_W$ is ample, $K_W^2 = 1$,  and $p_g(W) = 2, q(W) =0$. 
\end{exam}

\begin{exam}
\label{ex: E13}
 Let $P = \IF_1$. Consider a curve $\Delta = \Gamma_p + \Delta'\in |4C_\infty + 8 \Gamma|$ with the following properties
 \begin{itemize}
  \item $\Delta$ has a $[3,3]$ point at $p = \Gamma_p \cap C_\infty$ as in Remark \ref{rem: admissible33},
  \item $\Delta|_{C_\infty} = 3p + q$ for a second point $q$,
  
 \item along the fibre $\Gamma_q$ through $q$ the curve $\Delta$ has one additional singularity of type $A_1$ or $A_2$ in a point $q'\neq q$ and the local intersection number is $(\Delta, \Gamma_q)_{q'} = 2, 3$,
 \item $\Delta$ is smooth outside $p, q'$. 
     \end{itemize}
Again, such divisors form an open subset of a linear system, as we will now explain: consider the $\IZ^2$ graded ring $S = \IC[t_0, t_1, x_0, x_1]$ with grading
\[ \begin{pmatrix}
    t_0 & t_1 & x_0 & x_1 \\
    1& 1 & 0 & -1\\
    0& 0& 1& 1
   \end{pmatrix}
\]
This is the Cox ring of $\IF_1$ in the sense of \cite[Sect. 5.2]{cls2012}, but we only use that $(t_0: t_1)$ are the homogeneous coordinates on the base, that $C_\infty = \{x_1 = 0\}$ and that $H^0(4C_\infty + 8 \Gamma)$ is the linear space of polynomials of bidegree $\mat{4\\4}$.
We may normalise coordinates such that $p $ is defined by $x_1 = t_0 =0$, that $q $ is defined by $x_1 = t_1 =0$ and that  $q' $ is defined by $x_0= t_1 =0$.
Writing an equation for $\Delta$ as $f = t_0 f'$ we first write the closed conditions that $\Delta'$ has at least a $[2,2]$-point tangent to the fibre in $p$, contains $q$, and has at least a double point at $q'$, which similar to the previous example can be written as 
\[ f' \in V = \left\{ g \in (t_0^2,\, t_0x_1^2, \, x_1^4) \cap (t_0^2 t_1,\, x_1) \cap (t_1, \,x_0)^2 \mid \deg g  = \mat{3\\4}\right\}. \]
Since $V$ contains the elements $ t_0^2t_1^2x_0^4$, $ t_0^2t_1^6x_1^4$, $t_0^5x_0^3x_1$, and $t_1^8x_1^4$, the base locus of $V$ is supported at $\{p, q, q'\}$ and therefore the general element is smooth elsewhere by Bertini.
It is now sufficient to exhibit at each base point, elements of $V$ that have the correct behaviour at this point. To this end consider
\begin{itemize}
 \item $ \left(t_0^2t_1^2x_0^4- t_1^8x_1^4 \right)$, which gives the $[2,2]$ point at $p$, 
 \item $\left( \left( t_0x_0+ t_1(t_1-\lambda t_0) x_1\right) \left(t_0x_0+t_1(t_1-\mu t_0) x_1\right) x_1^2t_0^4 -t_1^5x_0^3x_1\right)$, which gives --  according to different choices for $\lambda, \mu$ -- the possibilities for $\Delta '$ at $q'$ in the second condition,
 \item $ t_0^3t_1x_0^4$, which is smooth at $q$.
\end{itemize}
A general linear combination of these elements gives an equation for $\Delta'$ with the prescribed behaviour and thus the $\Delta'$ with a node not tangent to the fibre in $q'$ form an open, hence irreducible subset of $V$.
The $\Delta'$ with worse behaviour at $q'$ are in the closure; note that the other given polynomials are considerably worse at $q'$, so appropriate choices for $\lambda$ and $\mu$ show that these cases exist.

     Now let $ \bar \theta \colon \bar X \to P$ be the double cover branched over $\Delta$. The preimage $2\bar{G} = \bar\theta^*\Gamma_p$ is a double fibre passing through an elliptic singularity of degree one of type $T_{2,3,n+6}$, where  we assume $n = 0$ or $n=1$.
Let $\bar\sigma\colon \hat X \to \bar X$ be the resolution of the elliptic point and of the $A_n$ surface singularity over the point $q'$. The preimage of $\Gamma_{q'}$ is then a singular fibre of type $I_2, I_3, III, IV$, compare Table \ref{tab: E13 sing fibre}.

Let $\bar \sigma \colon \hat X \to X$ be the contraction of the resulting $(-1)$-curve $\hat{G}={(\inverse{\bar\sigma})}_*\bar G$.
Then the induced map  $\pi\colon X \to \IP^1$ is a relatively minimal elliptic fibration. The strict transform of $C_\infty$ in $X$ is a bisection $E_1$,  which is a smooth rational curve with self-intersection $-3$. Together with the strict transform  $E_2 $ of the preimage of  $\Gamma_q$ we find a configuration of curves that can be contracted to give a surface $W$ with  a singularity of type $E_{13}$. 

It is straightforward to check that $K_W$ is ample, $K_W^2 = 1$,  and $p_g(W) = 2, q(W) =0$. 
\end{exam}

\begin{table}
 \caption{Possible configurations in the second fibre  for $E_{13}$
 }
  \label{tab: E13 sing fibre}
   \begin{tabular}{cccc}
 \toprule
  Kodaira fibre & in $\hat X$ and $X$ & in $\bar X$ & branch in $P$\\
  \midrule
$I_2$ & 
 \begin{minipage}[c]{\Breite}
 \centering
\begin{tikzpicture}
\draw[fibre] (-.1, .5) to[ bend right, looseness = 1.1] ++(0,1.25 );

\draw[E1] (-1, .25) 
node[above]{$E_1$}
to[out= 0, in =90] (0,0) to[in = 0,out = -90] (-1, -.25);
\draw[Ei] (0, -.25) to 
node[left, pos = 0.7] { $E_2$}
++(0,2);
   \end{tikzpicture} \end{minipage}
   &
 \begin{minipage}[c]{\Breite}
 \centering
\begin{tikzpicture}
\draw[E1] (-1, .25) to[out= 0, in =90]  (0,0) to[in = 0,out = -90] (-1, -.25);
\draw[Ei] (0, -.25) -- ++(0,1) to[out = 90, in = 135] ++ (0.3,.5) to [out = -45, in = 35] ++ (-.5, -.5);

 \fill[sing] (0, .865) node[above left]{$A_1$} circle (2pt);
   \end{tikzpicture} \end{minipage}
&
 \begin{minipage}[c]{\Breite}
 \centering
\begin{tikzpicture}
\draw[E1] (-1, 0) node[above] {$C_\infty$} -- ++ (1.25, 0);
\draw[Ei] (0, -.25) to
node[left, pos = 0.7] { $\Gamma_q$} ++(0,2);
\draw[branch] 
(-.2, 1.65) to[out = -45, in = 45] (0, .75) to[out = 225, in = 135, loop] ++(0,0)
node[right] {$q'$}
to[out = -45, in = 45] 
node[right] {$\Delta$}
(0,0)
node[below right] {$q$}
to
++ (-.2, -.2);
   \end{tikzpicture} \end{minipage}
\\
$I_3$ & 
 \begin{minipage}[c]{\Breite}
 \centering
\begin{tikzpicture}
\draw[fibre] (-.1, .5) to ++(0.5,.75);
\draw[fibre] (-.1, 1.5) to ++(0.5,-.75);
\draw[E1] (-1, .25) to[out= 0, in =90] (0,0) to[in = 0,out = -90] (-1, -.25);
\draw[Ei] (0, -.25) -- ++(0,2);
   \end{tikzpicture} \end{minipage}
   &
 \begin{minipage}[c]{\Breite}
 \centering
\begin{tikzpicture}
\draw[E1] (-1, .25) to[out= 0, in =90] (0,0) to[in = 0,out = -90] (-1, -.25);
\draw[Ei] (0, -.25) -- ++(0,1) to[out = 90, in = 135] ++ (0.3,.5) to [out = -45, in = 35] ++ (-.5, -.5);

 \fill[sing] (0, .865) node[above left]{$A_2$} circle (2pt);
\end{tikzpicture} \end{minipage}
&
 \begin{minipage}[c]{\Breite}
 \centering
\begin{tikzpicture}
\draw[E1] (-1, 0) -- ++ (1.25, 0);
\draw[Ei] (0, -.25) -- ++(0,2);
\draw[branch] 
(-.2, 1.65) to[out = -45, in = 0] (0, .75) to[out = 0, in = 45] (0,0) to ++ (-.2, -.2);
   \end{tikzpicture} \end{minipage}
\\
$III$ & 
 \begin{minipage}[c]{\Breite}
 \centering
\begin{tikzpicture}
\draw[fibre] (.5, 1.5) to[out  = 225, in= 90] ++(-.5, -.75) to[out  = -90, in= 135] ++(.5, -.75);
\draw[E1] (-1, .25) to[out= 0, in =90] (0,0) to[in = 0,out = -90] (-1, -.25);
\draw[Ei] (0, -.25) -- ++(0,2);
   \end{tikzpicture} \end{minipage}
   &
 \begin{minipage}[c]{\Breite}
 \centering
\begin{tikzpicture}
\draw[E1] (-1, .25) to[out= 0, in =90] (0,0) to[in = 0,out = -90] (-1, -.25);
\draw[Ei] (0, -.25) -- ++(0,1.5) to[out = -90, in = 135] ++ (0.25,-.5);
  \fill[sing] (0, 1.25) node[above left]{$A_1$} circle (2pt);
  \end{tikzpicture} \end{minipage}
&
 \begin{minipage}[c]{\Breite}
 \centering
\begin{tikzpicture}
\draw[E1] (-1, 0) -- ++ (1.25, 0);
\draw[Ei] (0, -.25) -- ++(0,2);
\draw[branch] 
(.5, 1.5) to[out = 225, in = 0] (0, 1.25) to[out = 180, in = 90, loop ] ++ (0,0) to[out = -90, in = 135] (0,0) to[out = -45] ++ (.2, -.2);
   \end{tikzpicture} \end{minipage}
\\$IV$ & 
 \begin{minipage}[c]{\Breite}
 \centering
\begin{tikzpicture}
\draw[fibre] (0, 1.25) ++ (-.2, -.2) to ++ (.6, .6);
\draw[fibre] (0, 1.25) ++ (-.2, .2) to ++ (.6, -.6);
\draw[E1] (-1, .25) to[out= 0, in =90] (0,0) to[in = 0,out = -90] (-1, -.25);
\draw[Ei] (0, -.25) -- ++(0,2);
   \end{tikzpicture} \end{minipage}
   &
 \begin{minipage}[c]{\Breite}
 \centering
\begin{tikzpicture}
\draw[E1] (-1, .25) to[out= 0, in =90] (0,0) to[in = 0,out = -90] (-1, -.25);
\draw[Ei] (0, -.25) -- ++(0,1.5) to[out = -90, in = 135] ++ (0.25,-.5);
\fill[sing] (0, 1.25) node[above left]{$A_2$} circle (2pt);
   \end{tikzpicture} \end{minipage}
&
 \begin{minipage}[c]{\Breite}
 \centering
\begin{tikzpicture}
\draw[E1] (-1, 0) -- ++ (1.25, 0);
\draw[Ei] (0, -.25) -- ++(0,2);
\draw[branch] 
(.25, 1) to[out = 120, in = -90, bend left] (0, 1.5) to[out = -90,  in = 135, looseness = 1.5] (0,0) to[out = -45] ++ (.2, -.2);
   \end{tikzpicture} \end{minipage}
\\
   
 \bottomrule
 \end{tabular}

\end{table}

\begin{exam}
\label{ex: E14}
 Let $P = \IF_1$. Consider a curve $\Delta = \Gamma_p + \Delta'\in |4C_\infty + 8 \Gamma|$ with the following properties
 \begin{itemize}
  \item $\Delta$ has a $[3,3]$ point at $p = \Gamma_p \cap C_\infty$  as in Remark \ref{rem: admissible33},
  \item $\Delta|_{C_\infty} = 3p + q$ for a second point $q$,
   \item the curve $\Delta$  is tangent to $\Gamma_q$, the fibre through $q$,
   \item the curve $\Delta$ has one additional singularity of type $A_1$ or $A_2$ in a point $q'\in \Gamma_q\setminus\{ q\}$, 
 \item $\Delta$ is smooth outside $p, q'$. 
     \end{itemize}
 Arguing as in Example \ref{ex: E13}, the set of curves $\Delta'\in |4C_\infty+7\Gamma|$ that satisfy the first three conditions and have a node at $q'$  is a (non-empty) open subset of  a linear subsystem $V$, whose closure  also contains the second possibility from Table \ref{tab: E14 sing fibre}. Thus these $\Delta$ form an irreducible family. 
     
Now let $ \bar \theta \colon \bar X \to P$ be the double cover branched over $\Delta$. The preimage $2\bar{G} = \bar\theta^*\Gamma_p$ is a double fibre passing through an elliptic singularity of degree one of type $T_{2,3,n+6}$, where  we assume $n = 0$ or $n=1$.
Let $\bar\sigma\colon \hat X \to \bar X$ be the resolution of the elliptic point and of the $A_n$ surface singularity over the point $q'$. The preimage of $\Gamma_{q'}$ is then a singular fibre of type $I_3$ or $  I_4$, compare Table \ref{tab: E14 sing fibre}.

Let $\sigma \colon\hat X \to X$ the contraction of the resulting $(-1)$-curve $\hat{G}={(\inverse{\bar\sigma})}_*\bar G$.
Then the induced map  $\pi\colon X \to \IP^1$ is a relatively minimal elliptic fibration. The strict transform of $C_\infty$ in $X$ is a bisection $E_1$,  which is a smooth rational curve with self-intersection $-3$. 
Together with the strict transform  $E_2+E_3$ of the preimage of  $\Gamma_q$ we find a configuration of curves that can be contracted to give a surface $W$ with  a singularity of type $E_{14}$. 

It is straightforward to check that $K_W$ is ample, $K_W^2 = 1$,  and $p_g(W) = 2, q(W) =0$. 
\end{exam}

\begin{table}
 \caption{Configurations in the second fibre for $E_{14}$
  }
  \label{tab: E14 sing fibre}
   \begin{tabular}{cccc}
\toprule
  Type & in $\hat X$ and $X$ & in $\bar X$ & branch in $P$\\
  \midrule
$I_3$ & 
 \begin{minipage}[c]{\Breite}
 \centering
\begin{tikzpicture}
\draw[fibre] (-.1, 1.7) to ++(0.4,-1.2);
\draw[Ei] (-.08, -.24) to node[right]{$E_2$} ++(.38,1.14);
\draw[Ei]  (.1,1.6) to[  in = 135, out  = -135] 
node[left, pos = .4] {$E_3$}
(0,0) to ++(-45:0.25);
\draw[E1] (.5, 0) to ++(-1.25,0) to[in = 180,out = 180] node[left]{$E_1$} 
++ (0, .25)
to ++ (.25,0);
   \end{tikzpicture} \end{minipage}
   &
    \begin{minipage}[c]{\Breite}
 \centering
\begin{tikzpicture}
\draw[Ei]  (-.1,1.6) to[  in = 45, out  = -45] (0,0) to ++(-.1, -.1);
\draw[Ei]  (.1,1.6) to[  in = 135, out  = -135] (0,0) to ++(.1, -.1);
\draw[E1] (.5, 0) to ++(-1.25,0) to[in = 180,out = 180] ++ (0, .25) to ++ (.25,0);
\fill[sing] (0, 1.5) node[left]{$A_1$} circle (2pt);
   \end{tikzpicture} \end{minipage}
&
 \begin{minipage}[c]{\Breite}
 \centering
\begin{tikzpicture}
\draw[E1] (-1, 0)  node[above] {$C_\infty$} -- ++ (1.25, 0);
\draw[Ei] (0, -.25) to node[left, pos = 0.8] { $\Gamma_q$}  ++(0,2);
\draw[branch] 
(.2, 1.65) to[out = -45, in = 45] 
node[right]{$\Delta$}
(0, 1) to[out = 225, in = 135, loop] ++(0,0)
node[right] {$q'$}
to[out = -45, in = 90] (0,0)
node[below left] {$q$}
to[out = -90, in = 135] ++ (.2, -.2);
   \end{tikzpicture} \end{minipage}
\\
$I_4$ & 
 \begin{minipage}[c]{\Breite}
 \centering
\begin{tikzpicture}
\draw[fibre] (-.1, 1.7) to ++(0.4,-1.2);
\draw[fibre] (.1, 1.7) to ++(-0.4,-1.2);
\draw[Ei] (-.1, -.3) to ++(.4,1.2);
\draw[Ei] (.1, -.3) to ++(-.4,1.2);
\draw[E1] (.5, 0) to ++(-1.25,0) to[in = 180,out = 180] ++ (0, .25) to ++ (.25,0);
   \end{tikzpicture} \end{minipage}
   &
    \begin{minipage}[c]{\Breite}
 \centering
\begin{tikzpicture}
\draw[Ei]  (-.1,1.6) to[  in = 45, out  = -45] (0,0) to ++(-.1, -.1);
\draw[Ei]  (.1,1.6) to[  in = 135, out  = -135] (0,0) to ++(.1, -.1);
\draw[E1] (.5, 0) to ++(-1.25,0) to[in = 180,out = 180] ++ (0, .25) to ++ (.25,0);
\fill[sing] (0, 1.5) node[left]{$A_2$} circle (2pt);

   \end{tikzpicture} \end{minipage}
&
 \begin{minipage}[c]{\Breite}
 \centering
\begin{tikzpicture}
\draw[E1] (-1, 0) -- ++ (1.25, 0);
\draw[Ei] (0, -.25) -- ++(0,2);
\draw[branch] 
(.2, 1.65) to[out = -45, in = 0] (0, 1)  to[out = 0, in = 90] (0,0) to[out = -90, in = 135] ++ (.2, -.2);
   \end{tikzpicture} \end{minipage}
\\
\bottomrule
 \end{tabular}
 \end{table}

\begin{rem}
 In the above examples, the branch divisor has been chosen general enough to guarantee that the surface $W$ constructed has a unique singularity and ample canonical bundle. 
 
 One could easily relax the assumptions to allow some additional ADE-singularities on $\Delta$ and then on $W$, or to allow the multiple fibre to be reducible, but the language becomes more cumbersome as we have to keep track of additional $(-2)$-curves. 
\end{rem}

\begin{thm}\label{thm: type E}
 Let $W$ be a Gorenstein surface with $K_W^2 = 1$, $p_g(W) = 2$ and $q(W)=0$ and a unique singular point of type $E_{12}, E_{13}$, or  $E_{14}$. Then $W$ arises as in Example \ref{ex: E12}, Example \ref{ex: E13}, or Example \ref{ex: E14}.
 
For each type of singularity there is  one irreducible family.
\end{thm}

\subsection{Proof of Theorem \ref{thm: type E}}\label{sect: proof of En}
Now let $W$ be a surface with $K_W$ ample, $K_W^2 = 1$, $p_g(W) = 2$ and $q(W)=0$ and with  a unique singular point of type $E_{12}, E_{13}$, or  $E_{14}$. We study the situation of \eqref{Diagram1} via a series of Lemmata. 

\begin{lem}\label{classification of E_i} The surface $X=S$ is a minimal elliptic surface with $p_g(X) =1$ and $q(X) = 0 $, so there is a minimal elliptic fibration $\pi \colon X \to \IP^1$.
\end{lem}
\begin{proof} 
From Table \ref{tab: info En} we have $E^2 = -1$ and by adjunction $K_X.E = 1$, so $K_X^2 =0$. Note that by Proposition \ref{prop: min res p_g=2}, we have $p_g(X)=1$ and $q(X)=0$, and so $\kappa(X)\geq 0$.

Assume first that $X$ has Kodaira dimension $\kappa(X)=0$. Since $K_X^2=0$, we obtain that $X$ is a minimal model, and then $K_X$ is numerically trivial. But $K_X.E=1$, a contradiction.

Now assume $X$ is of general type. Since $K_X^2=0$ we have that $X$ is not a minimal surface and by using the inequality $0 < K_S^2\leq K_W^2=1$ from Lemma \ref{lem: basic stuff}, we have that there exists a unique $(-1)$-curve $G$ in $X$ and that $K_S^2=1$. Hence,
\[f^*K_S+G = K_X=\varphi^*K_W-E,\]
so, it follows that 
\[K_W^2=(\varphi^*K_W)^2=(\varphi^*K_W).f^*K_S.+(\varphi^*K_W).E+(\varphi^*K_W).G,\]
and then
\[K_W^2=K_S^2+K_S.(f_*E)+K_W.(\varphi_*G),\]
where each term is non-negative. Since $K_W^2=K_S=1$ then we have that $K_S.(f_*E)=K_W.(\varphi_*G)=0$. This is impossible, because $G$ is not $\phi$-exceptional and $K_W$ is ample. 
Thus, the surface $X$ has $\kappa(X)=1$.

Now, since $K_X^2=0$, the surface $X=S$ is minimal; its invariants have been computed in Proposition \ref{prop: min res p_g=2}.
The base of the fibration has genus zero, because $q(X) = 0 $. 
\end{proof}

We now study the elliptic fibration $\pi \colon X \to\IP^1$ provided by Lemma \ref{classification of E_i} and we follow closely the arguments of \cite[Sect. 3.1.1]{do-rollenske22}, where a similar problem was considered. . Our basic reference for the theory of elliptic fibrations is \cite{Friedman} and an important invariant is the line bundle $L:= \pi_*\omega_{X/\IP^1}$.
\begin{lem}\label{lem:bisection}
If the singularity $p\in W$ is of type $E_{12}, E_{13}$ or $E_{14}$, then the minimal resolution $X$ contains a curve $E = \sum E_i$ as in Table \ref{tab: info En}. In our situation, the following properties hold:
\begin{enumerate}
\item The curve $E_{1}$ is a bisection and  the fibration $\pi$ has exactly one double fiber $2F$,  where $F$ denotes the reduced multiple fiber of $\pi$. In addition,
\[L:= \pi_*\omega_{X/\IP^1} = \ko_{\IP^1}(2) \text{ and } K_X\equiv F.\]
\item $F$ is irreducible and either smooth elliptic (type $I_0$) or a nodal rational curve (type $I_1$). 
 \item In case  $E_{13}$ there is one singular  fibre of type $I_2$, $I_3$, $III$, or $IV$ such that $E_2$ is one of the curves in the fibre (compare Table \ref{tab: E13 sing fibre}).
 \item In case  $E_{14}$ there is one singular  fibre of type $I_3$ or $I_4$ containing the curves $E_2$ and $E_3$ (compare Table \ref{tab: E14 sing fibre}).
 \item All other singular fibres have at most two irreducible components, thus are of type $I_1, I_2, II, III$.
\end{enumerate}
\end{lem}
\begin{rem}
 Noether's formula gives that $c_2(X) = 23$, so there is plenty of room for the required singular fibres. 
\end{rem}

 \begin{proof} 
 Let us denote by  $F_i$ the reduced multiple fibres of $\pi$ with multiplicities $m_i$. Recall that they are of type $I_{k_i}$, so the non-reduced Kodaira fibres cannot be multiple \cite[Lemma 5, Section 7]{Friedman}.
By the canonical bundle formula \cite[Theorem 15, Section 7]{Friedman} 
we have
\begin{equation}\label{fiber canonical formula}
K_{X} = \pi^*\left(K_{\mathbb{P}^1}+L\right) + \sum_{i} (m_i-1) F_i.
\end{equation}

By \cite[Lemma 14, Section 7]{Friedman} we have $1=p_g(X)=\deg L+g(\mathbb{P}^1)-1$, and then $\deg L=2$. Thus, we have $L = \pi_*\omega_{X/\IP^1} = \ko_{\IP^1}(2)$, and so  \eqref{fiber canonical formula} gives
\begin{equation}\label{d=2 fiber canonical formula}
     K_X\equiv \sum_{i=1}^k (m_i-1) F_i.
\end{equation}
 
 Now, we know that $E_1$ cannot be contained in a fibre because it is reduced, irreducible, and $E_1^2\neq -2$. Thus, $E_1$ is a $d$-multisection with $d\geq 1$. 
The self-intersection numbers and the adjunction formula give
\[
 1 =K_X. E= K_X. E_1 =  \sum_{i=1}^k (m_i-1) F_i. E_1 \geq \sum_{i = 1}^k (m_i-1)\geq k
\]
and we see that $k = 1 = m_1-1$, that is, there is a unique multiple fibre of multiplicity two and $E_1$ is a bisection.

On the other hand, by Proposition \ref{prop: min res p_g=2}, we have that $p_g(X)=1$ and the only effective canonical divisor is of the form
\[ K_X = F +\sum a_i E_i\]
for an irreducible and reduced curve $F$ not contained in $E$. To show $(i)$, we prove that $a_i = 0 $, that is, the multiple fibre is irreducible, thus of type $I_0$ or $I_1$.

Clearly $a_1 = 0$, because $E_1$ is not contained in a fibre, which concludes the case $E_{12}$.

The curves $E_2, E_3$, if present, are $(-2)$-curves and have to be contained in the a fibre of the elliptic fibration; in case $E_{14}$ in fact both are  in the same fibre,  because they intersect. They cannot be contained in the unique double fibre, because $E_1.(E-E_1) = 2$ in cases $E_{13}$ and $E_{14}$, and $E_1$ is a bisection. Thus $a_i = 0$ for $i>1$ as well and we have proved $(i)$. 

Finally, since $f^*K_W = F+E$ is nef, and positive on all curves except the curves $E_i$, we see that every fibre not containing components of $E$ can have at most two components, which both intersect the bisection $E_1$, which implies $(iv)$.

A simple case by case examination shows that the necessary configurations for $E_{13}$ and $E_{14}$ can only fit into the kind of fibres already listed in Table \ref{tab: E13 sing fibre} and Table \ref{tab: E14 sing fibre}, proving $(ii)$ and $(iii)$. 
\end{proof}

\begin{lem}
There exists a point $x\in F$ such that $x\not\in E_1$ and
$(K_{X/\mathbb{P}^1} +E_1)|_F$ is linearly equivalent to $x$.
\end{lem}
\begin{proof}
Note that $\ko_F(F)$ is a non-trivial $2$-torsion bundle on $F$ by Lemma \ref{lem:bisection}. We have
  \[K_{X/\IP^1} = K_{X} - \pi^*K_{\IP^1} = \pi^*L + F,\] thus
  \[\left(K_{X/\IP^1} +E_1\right)|_F = \left(\pi^*L +F + E_1\right)|_F = \left(E_1+F\right)|_F\]
  which has degree $1$ and thus is linearly equivalent to a unique effective divisor $x$. We have $x\not \in E_1$ because $F|_F$ is non-trivial.
\end{proof}
Now we consider $\sigma \colon \hat X = \mathrm{Bl}_{x}(X)\to X $, the blow up in the point $x$,  and denote the exceptional curve over $x$ with $\hat{G}$. Let $\hat \pi = \pi \circ \sigma\colon \hat X \to \mathbb{P}^1$ be the induced fibration and let $\hat E_i$ respectively $\hat F$ be the strict transforms of $E_i$ and  $F$ in $\hat X$. 

\begin{prop}\label{prop: Xbar}
We consider the line bundle 
  \begin{equation}\label{eq: M} \hat{M}  = K_{\hat X/{\IP^1}} + \hat E_1 -2\hat{G}= K_X + 4\hat F + 2\hat G + \hat E_1 = \hat \pi ^* L +\hat E_1 +\hat F.
  \end{equation}
on $\hat X$.
Then the following holds:
\begin{enumerate}
\item $\hat{M}|_{\hat F} \isom  \ko_{\hat F}$ and $\ko_{\hat F}(\hat F) \isom \ko_{\hat F}(- \hat E_1)$
 \item $\hat M$ is big and nef and semi-ample. 
   If $C$ is an irreducible curve in $\hat X$ then $\hat{M}.C = 0 $ if and only if
    $C = \hat F$ or $C$ is a component of a singular non-multiple fibre that does not intersect $\hat E_1$ (compare Tables \ref{tab: E13 sing fibre} and \ref{tab: E14 sing fibre}).
\item 
There is a normal projective surface $\bar X$ together with an ample line bundle $\bar M$ fitting in the diagram 
 \[\begin{tikzcd} \hat X \arrow{rr}{\bar \sigma} \arrow{dr}[swap]{\hat \pi}&& \bar X\arrow{dl}{\bar\pi}\\ & \IP^1\end{tikzcd}\]
 such that
\begin{enumerate}
 \item $\bar \sigma$ is birational and contracts exactly the curves on which $\hat M$ is numerically trivially listed in $(ii)$.
 
Therefore $\bar X$ has  one  elliptic singularity of type $T_{2,3 n}$ with $n= 6,7$ at the point $\bar p = \bar \sigma(\hat F)$ and possibly  one additional $A_1$ or $A_2$ singularity in case $E_{13}$ or $E_{14}$.

 \item We have $\hat M = \bar \sigma^*\bar M $ and $K_{\bar X}\isom \bar \pi ^* \ko_{\IP^1}(1)$.
\end{enumerate}
\end{enumerate}
\end{prop}
\begin{proof}
We start with $(i)$. 
From  $\hat{M}=K_{\hat{X}/\IP^1} + \hat{E_1} - 2\hat{G}$ we get 
    \[
      (K_{\hat{X}/\IP^1} + \hat{E_1} -2\hat{G}) = \sigma^{*}\left(K_{\bar{X}/\IP^1} + E_1\right) -  \hat{G}.
    \]
    The strict transform of $F$ is $\hat{F} \isom F$. Then
    $(K_{X/\IP^1} +E_1)|_{F} = x$,
    \begin{align*}
      \hat{M}|_{\hat{F}}& =\left(\sigma^{*}(K_{\bar{X}/\IP^1}+E_1) - \hat{G}\right)|_{\hat{F}}\\
      &= x - x = 0.
    \end{align*}
    Thus $\hat{M}|_{\hat F} = \ko_{\hat F}$ and $\ko_{\hat F}(\hat F) \isom \ko_{\hat F}(- \hat E_1)$.

Now we turn to $(ii)$.
The second expression for $\hat{M}$ in \eqref{eq: M} shows that $\hat{M}$ is effective and we only have to compute that
 \[ \hat{M}.\hat F = 0 , \; \hat{M}.\hat E_1 = 4, \hat{M}. (\text{gen. fibre}) = 2 , \text{ and } \hat{M}^2 =  8.\]
  If $\hat{M}.C=0$ then $C$ cannot be horizontal  or an irreducible fibre, so is contained in a reducible singular fibre. Since $\hat{M}.\hat G = 1 $, the curve $C$ is one of the listed curves.

To prove semi-ampleness we look at
\[ 2\hat M - \hat F= K_{\hat X} + 2\left(\hat E_1 + \hat F +\hat\pi^*\ko_{\IP^1}(1)\right).
\]
Since $2\left(\hat E_1 + \hat F +\hat\pi^*\ko_{\IP^1}(1)\right)$ has even  positive  degree on every irreducible curve  and has self-intersection $8$,  Reider's theorem  \cite[IV.11.4]{BHPV} applies to  shows that  $2\hat M - \hat F$ has no base-points, so the base-locus of $2 \hat M$ is contained in $\hat F$. 

The in the same way we check that $\hat M -\hat F - K_{\hat X} = 3\hat F + 2\hat G + \hat E_1$ is  big and nef, so that by  Kodaira vanishing the sequence
\[ 0 \to H^0(\hat X, \hat M - \hat F) \to H^0(\hat X, \hat M ) \to H^0(\hat X, \hat M|_{ \hat F})  \to 0 \]
is exact. 
Since by $(i)$ the restriction of $\hat M$ to $\hat F$ is trivial, we can find a section of $\hat M$, and hence also of $2\hat M$ not vanishing anywhere on $\hat F$, so there are also no base points of $2\hat M$ on $\hat F$.

For $(iii)$ we take a suitable multiple of $\hat M$ and get the map $\bar \sigma \colon \hat X \to \bar X$, that contracts exactly the curves on which $\hat M$ is not ample, which are listed in $(ii)$.

The surface $\bar X$ still fibres over $\IP^1$, because we only contract curves in the fibres.
The simple elliptic point $\bar p$ is a strictly log-canonical singularity, so we compute 
\[ \bar \sigma^*K_{\bar X} = K_{\hat X} + \hat F = ( \sigma^* K_X + \hat G) + \hat F = (\sigma^*F +\hat G) + \hat G = 2(\hat F + \hat G),\]
which is exactly the class of a fibre of $\hat \pi$. Thus we get the formula for the canonical bundle. 

 We can define $\bar M:=\bar\sigma_*\hat M$, which is a line bundle because we have seen in the proof of $(ii)$ hat $\hat M$ is trivial in an open neighbourhood of $\hat F$ and the original description \eqref{eq: M} shows that it is trivial also in the neighbourhood of any other curve that is contracted.
\end{proof}

\begin{prop}\label{prop: double cover Hirzebruch}
 Let $(\bar X, \bar M)$ be as in Proposition \ref{prop: Xbar}. Then 
 \begin{enumerate}
  \item $\bar\pi_*\bar M $ is a locally free sheaf of rank two,

  \item the map $\bar \pi^*\bar \pi_* \bar M \to \bar M$ is surjective,
  \item the induced finite morphism $\bar \theta \colon \bar X \to P = \IP(\bar \pi _*\bar M)$ is a double cover. If $\Gamma_p$ is the ruling through $p=\bar\theta (\bar{p})$ then the branch divisor is of the form $\Delta = \Gamma_p + \Delta' \in |2(\Gamma- K_P)|$.
  \item $\Delta$ has a non-degenerate $[3,3]$-point at $p$, in such a way that 
  $\Gamma_p.\Delta' =4$ and $\Delta'$ as an $A_3$ or $A_4$ singularity at $p$. 
  \item $\bar E_1 = \bar \sigma (\hat E_1)$ is not contained in the branch locus of $\bar \theta$ and maps to a section of $P \to \IP^1$. 
 \end{enumerate}
\end{prop}

We start with some  Lemmata, controlling the restriction of $\bar M$ to the fibres of $\bar \pi$. 
The first result is taken from  \cite[Lem. 2.4]{CFPR23} but we include their proof for the convenience of the reader. 
\begin{lem}\label{vanishingH^1}
 Let $C$ be a  reduced Gorenstein curve and $\kf$ a rank one  torsion-free sheaf on $C$. 
 For a subcurve $B$ consider the sheaf $\kf|_{B}^{[1]}: = \left(\kf\tensor \ko_B\right)/(\text{torsion})$ and define its degree by the Riemann--Roch formula $\deg\left(\kf|_{B}^{[1]}\right) : = \chi(\kf|_{B}^{[1]}) - \chi(\ko_B)$.
 
 If $\deg\left(\kf|_{B}^{[1]}\right) \geq 2p_a(B)-1$  for every  subcurve $B\subseteq C$, then $H^1(C,\kf)=0$. 
\end{lem}

\begin{proof}
The proof follows by the arguments used in \cite[Thm.1.1]{CFHR}. 
By Serre duality $H^1(C,\kf)^{\vee} \cong \Hom (\kf, \omega_C)$. Assume it is not zero and 
pick any nonzero map
$\varphi \colon \kf \to \omega_C$. By \cite[Lemma 2.4]{CFHR}, $\varphi $ comes from a generically surjective 
map 
$\kf|_{B}\to\omega_B$ for a subcurve $B\subseteq C$  and   yields  an injection $\kf|_{B}^{[1]} \into \omega_B$  whose 
cokernel has
finite length. Therefore   $\chi\left(\kf|_{B}^{[1]}\right)\leq\chi(\omega_B)$, thus
$\deg\left(\kf|_{B}^{[1]}\right)   \leq 2p_a(B)-2$, against the assumptions.  
\end{proof} 

As an application we get the following.
\begin{lem}\label{lem: bisection on non-multiple}
 Let $C = \sum_i m_i C_i$ be a non-multiple fibre of a relatively minimal elliptic fibration as classified by Kodaira and let $M$ be an effective Cartier divisor of degree two on $C$. Then $h^0(C, M) = 2$ and $M$ is base-point-free. 
\end{lem}
\begin{proof}
Recall that while $p_a(C) = 1$, we have $p_a(B)=0$ for every proper subcurve $B<C$. 
It is now straightforward to check that for any point $p\in C$ and for any subcurve $B<C$ both $M$ and $\gothm_p(M)$ satisfy the assumptions of Lemma \ref{vanishingH^1} and thus have vanishing first cohomology. Thus $h^0(C, M) = \chi(M) =2$ and the associated linear system has no base-points. 
\end{proof}
More specific to our situation we have the following.

\begin{lem}\label{lem: bisection on multiple fibre}
 The fibre of $\bar \pi$ through the elliptic point $\bar{p}\in \bar X$ is a non-reduced Gorenstein curve with trivial canonical bundle, which we denote by $2\bar G$. The reduction $\bar G$ is a smooth curve of genus zero and 
 \[ H^0(\bar M|_{2\bar G})\isom H^0(\bar M|_{\bar G})\]
 has two sections that define a base-point free pencil on $2\bar G$. 
 \end{lem}
\begin{proof}
The scheme-theoretic fibre is by definition a Cartier divisor in $\bar X$ with trivial normal bundle and the adjunction formula gives $\omega_{2\bar G} = \ko_{2\bar G}$. As the bijective map $\bar\sigma|_{\hat G}\colon \hat G \to \bar G$ is an isomorphism except possibly at the point $\bar p\in \bar G$, it is enough to study the local situation close to the contracted curve $\hat F$.

By \cite[4.25]{reid97}, the analytic germ at $\bar p$ is locally isomorphic to the spectrum of 
\begin{equation}\label{eq: ring at ellipic point}
 R(\hat F, \ko_{\hat F}(-\hat F))\isom \IC[a,b,c]/(c^2 + f_6),
\end{equation}
where $\deg(a,b,c) = (1,2,3)$ and $f_6$ is weighted homogeneous of degree $6$ not divisible by $a$. The inclusion of  $\IC[a^2, b]$ into this ring defines a double cover $\psi\colon \hat F \to \IP^1$ branched over the four zeros of $a^2f_6$. 

Considered as sections on $\hat F$, we have $a\in H^0(\ko_{\hat F}(-\hat F))$, thus vanishing at $\hat E \cap \hat F$. 
For $b$ we can choose any section $\ko_{\hat F}(-2\hat F) \isom \ko_{\hat X} (2\hat G)|_{\hat F}$ linear independent of $a^2$, so we can choose $b$ to be the section defining $2\hat G$ near $\hat F$. Thus $b$ vanishes twice at $\hat x = \hat G \cap \hat F$. Therefore, $b$ vanishes at a ramification point of $\psi$, that is, $f_6$ is divisible by $b$ with this choice.

Then $2\bar G$ is defined by $\IC[a,b,c]/(c^2 + f_6, b)\isom \IC[a,c]/(c^2)$, a non-reduced curve with smooth reduction as claimed.

Now we consider the decomposition sequence of the curve (compare \cite[II.1]{BHPV} or \cite{CFHR}) twisted by $\bar M$,
\[
 0 \to \ko_{\bar G}(\bar M-\bar G) \to \ko_{2\bar G}(\bar M) \to \ko_{\bar G}(\bar M) \to 0
\]
Since $\bar G \isom \IP^1$ and $\chi (\ko_{2\bar G}) = \chi(\ko_{\bar G}) + \chi(\ko_{\bar G}(-\bar G)) = 0$ we have $\deg \ko_{\bar G}(-\bar G) = -2$. For  $\bar M$ we have $\bar M.\bar G = 1$ and then the long exact sequence in cohomology shows that
$H^0(\ko_{2\bar G}(\bar M) = H^0( \ko_{\bar G}(\bar M))$ defines a base-point free pencil on $2\bar G$. 
\end{proof}

\begin{proof}[Proof of Proposition \ref{prop: double cover Hirzebruch}]
From Lemma \ref{lem: bisection on non-multiple} and Lemma  \ref{lem: bisection on multiple fibre} we see that for all $b\in \IP^1$ we gave  $h^0(\inverse {\bar \pi}(b), \bar M_{\inverse {\bar \pi}(b)}) = 2$. Therefore by the base-change theorem
 \cite[I.8.5]{BHPV} the pushforward $\bar\pi_*\bar M$ is locally free of rank two, which proves $(i)$. By loc. cit. the sections define a base-point free pencil on every fibre, so we also have $(ii)$. By standard theory of projective bundles \cite[II.7]{Hartshorne} we get the morphism $\bar\theta$. It is finite, because $\bar M$ is ample by Proposition \ref{prop: Xbar} and of degree two, because it has degree two on the general fibre. 
 
 If $\Gamma$ is a fibre of $\bar\pi$ then again by Proposition \ref{prop: Xbar} we have by the Hurwitz formula
 \[ \bar\theta ^* \Gamma = K_{\bar X} = \bar\theta^*\left(K_P + {\textstyle\frac12} \Delta\right).\]
 Since $2\bar G$ is a double fibre, the fibre $\Gamma_p$ has to be contained in the branch locus, so $\Delta = \Gamma_p + \Delta'$. 
 The contraction of $\hat F$ to an elliptic singularity of degree one forces $\Delta$ to have a $[3,3]$ point at the image point as described in Remark \ref{rem: admissible33}. 
 Since in addition $\hat F$ is irreducible, the elliptic point has to be of type $T_{2,3,6}$ or $T_{2,3,7}$, which gives the possibilities described for the branch divisor. 
  This proves $(iii)$ and $(iv)$. 

 For the last item, we just have to note that from Lemma \ref{lem:bisection} $E_1$ is a bisection, so $\bar{E_1}$ is, too. It remains to show that $\bar E_1$ is not in the branch locus of the map  $\bar \theta \colon \bar X \to P$. To see this we consider a general fibre $\bar X_t$ of $\bar \pi$. The restriction $\bar E_1|_{\bar X _t} = \bar r_1 + \bar r_2$ consists of  two different points, because $\bar E_1$ is a reduced curve and thus cannot be tangent to every fibre.
  The restriction of the map $\bar\theta_t = \bar \theta|_{\bar X_t}\colon \bar X_t \to P_t \isom \IP^1$ is given by the sections in $H^0(\ko_{\bar X_t}(\bar r_1 + \bar r_2))$ and thus $r_1+ r_2 = \bar E_1|_{\bar X _t}$ is one of the fibres of $\bar \theta _t$. Hence $r_i$ is not in the branch locus of $\bar \theta_t$ and $\bar E_1$ is not in the branch locus of $\bar\theta$.
   \end{proof}

\begin{lem}\label{lem: proj bundle}
We have $\hat\pi_*\hat{M} \isom \ko_{\mathbb{P}^1}(2)\oplus \ko_{\mathbb{P}^1}(p_a(E_1)+1)$ and thus $P = \mathbb{P}(\hat\pi_*\hat{M})\isom \mathbb{F}_{1-p_a(E_1)}$ and $\bar \theta (\bar E_1 )= C_\infty$ is a section of minimal square, unique if $p_a(E_1)=0$. 
\end{lem}

\begin{proof}
  First note that  $\sigma^* F = \hat F + \hat{G}$, thus
  \[\sigma_* \ko_{\hat X}( \hat F ) =  \sigma_*\ko_{\hat X}( \sigma^*F - \hat{G}) =  \ki_{x}(F)\subset \ko_{X}( F),\]
so that by \cite[Ch.7, Exercise 2]{Friedman} we have
 \[ \hat\pi_* \ko_{\hat X} (\hat F  )\subset \pi_*\ko_{X} (F) = \ko_{\IP^1}.\]
  Since the left hand side has a global section, this inclusion is an equality. This implies that the pushforward $\hat\pi_*(\hat \pi^*L + \hat F  ) = L$, where $L := \pi_*\omega_{X/\IP^1}$.

  Now consider the exact sequence
   \[ 0 \to \hat\pi^*L + \hat F  \to \hat{M} \to \hat{M}|_{\hat E_1} = K_{\hat E_1/{\IP^1}}\to 0.\]
Applying $\hat\pi_*$ we get by the projection formula, the above computation and using both descriptions of $\hat{M}$ from \eqref{eq: M},
  \begin{equation}\label{eq: ex seq type A}
    \begin{split}
      & 0\to L = \hat\pi_*\left(\hat \pi^*L + \hat F  \right) \to \hat\pi_* \hat{M} \to \hat \pi_*K_{\hat E_1/{\IP^1}} 
      \to\\
      &\qquad  R^1\hat\pi_*\left(K_{\hat X/{\IP^1}} -2G\right) \to R^1\hat\pi_* \hat{M}\to ...
    \end{split}
  \end{equation}

   By relative duality  \cite[\href{https://stacks.math.columbia.edu/tag/0AU3}{Section 0AU3}]{stacks-project} we have
   \[ \hat{\pi}_*K_{\hat E_1/{\IP^1}}  = \left(\hat\pi_*\ko_{\hat E_1}\right)^\vee = \ko_{\IP^1} \oplus \ko_{\IP^1}(p_a(E_1)+1),\]
   where we use that $\chi(\ko_{E_1}) = 1 - p_a(E_1)$, or equivalently, that the double cover $E_1 \to \IP^1$ has $2p_a(E_1)-2$ branch points.
   
   Using duality again we compute
      \[R^1\hat\pi_*\left(K_{\hat X/{\IP^1}} -2\hat{G}\right) = R^1\hat\pi_*\shom\left(\ko_{\hat X}  (2\hat{G}) ,  K_{\hat X/{\IP^1}}\right) = \shom\left(\hat\pi_* \ko_{\hat X}  (2\hat{G}), \ko_{\IP^1}\right)\isom \ko_{\IP^1},\] where the last identification is proved by pushing forward the exact sequence
  \[ 0 \to \ko_{\hat X} \to\ko_{\hat X} (2\hat{G}) \to \ko_{2\hat{G}}( 2\hat{G})\to 0 .\]
   Indeed, $\hat\pi_* \ko_{2\hat{G}}(2\hat{G}) $ is a skyscraper sheaf supported at the point $\hat \pi(\hat{G})$ with stalk $H^0(2\hat{G}, \ko_{2\hat{G}}(2\hat{G}))$. This is zero, because $\hat{G}^2 = -1$. 
 
  Repeating this for $\hat{M} = K_{\hat X/{\IP^1}} - 2\hat{G} + \hat E_1$ we get
  \[R^1\hat\pi_* \hat{M} = \left(\hat\pi_*\ko_{\hat X}( 2\hat{G}-\hat E_1)\right)^\vee=0\]
  because $\ko_{\hat X}(2\hat{G}-\hat E_1)$ restricted to the general fibre has negative degree and thus no sections, so $\hat\pi_*\ko_{\hat X}(2\hat{G}-\hat E_1)$ is a torsion sheaf and its dual is trivial.

  Therefore the  sequence \eqref{eq: ex seq type A} is isomorphic to
  \[ 0 \to \ko_{\IP^1}(2) \to \hat\pi_*\hat{M} \to \ko_{\IP^1} \oplus \ko_{\IP^1}(p_a(E_1)+1) \to \ko_{\IP^1}\to 0,\]
  and since $\Hom(\ko_{\IP^1}(p_a(E_1)+1), \ko_{\IP^1}) = 0$ and 
  \[\Ext^1(\ko_{\IP^1}(p_a(E_1)+1), \ko_{\IP^1}(2)) = H^1(\IP^1, \ko(1-p_a(E_1))) = 0\] we have $\hat\pi_*\hat{M} \isom \ko_{\IP^1}(2)\oplus \ko_{\IP^1}(p_a(E_1)+1)$.

Now we have to compute the class of the section  (Proposition \ref{prop: double cover Hirzebruch}) $\bar\theta (\bar E_1) \sim C_\infty + k\Gamma$ in $P$.

First of all we look once again  at the local model near $p$ described by the ring \eqref{eq: ring at ellipic point}. In this model the curve $\bar E_1$ is locally described by the vanishing of the section $a$, thus is a cuspidal curve, since $f_6\equiv b^3\mod a$.
Since $\hat E_1$ was smooth at the intersection with $\hat F$,  we have $p_a(\bar E_1) = p_a (\hat E_1) + 1$ where $p_a(\hat{E_1})=p_a(E_1)$. By the Hurwitz formula, the double cover $\bar E_1 \to \bar \theta(\bar E_1) \isom \IP^1$ has 
\[\Delta.\bar \theta(\bar E_1) = 2p_a(\bar E_1) +2 = 2p_a( E_1) +4\]
branch points. 
We then compute using Proposition \ref{prop: Xbar}\,$(iii)\,(b)$
\begin{multline*} 
1 = \bar \theta(\bar E).\Gamma = \bar\theta(\bar E_1).\left ( K_P+ \frac12 \Delta\right) 
= (C_\infty + k\Gamma). \left( -2C_\infty -(2+1-p_a(E_1))\Gamma\right) + p_a(E_1) +2\\
= -2(p_a(E_1) - 1)- (3- p_a(E_1)) -2k +p_a(E_1) +2 =-2k+1,
\end{multline*}
so $k =0 $ and $\bar \theta (\bar E_1) \sim C_\infty$. 

If $p_a (E_1) = 0$ the negative section is unique, if $p_a(E_1)=1$, then the image is the horizontal section passing through the $[3,3]$ point of the branch divisor. 
\end{proof}

To conclude the proof of Theorem \ref{thm: type E}, we now analyse the individual cases using  Proposition \ref{prop: double cover Hirzebruch} and Lemma \ref{lem: proj bundle}.
\begin{prooflist}
\item[Type $E_{12}$] In this case $p_a(E_1) = 1$,so  \[\begin{tikzcd}
 \hat X \rar\arrow{drr}&  \bar X \rar{\bar \theta}\arrow{dr}{\hat \pi}& P = \IP^1\times \IP^1 \arrow{d}\\ && \IP^1  
   \end{tikzcd}
 \]
 with 
 $\bar \theta(\bar E_1) = C_\infty$ and $\Delta = \Gamma_p + \Delta'\in |4C_\infty +6\Gamma|$, which has a $[3,3]$-point at the intersection of $\Gamma_p$ and $C_\infty$. 
 
 Since $\bar E_1$ is a curve of arithmetic genus two with two cusp singularities, $\Delta|_{C_\infty}$ consists of two points with multiplicity three. 
 
In other words, $W$ arises exactly as in  Example \ref{ex: E12}.
 \item[Types $E_{13}$ and $E_{14}$] In these cases $p_a(E_1) = 0 $, so we have 
  \[\begin{tikzcd}
 \hat X \rar\arrow{drr}&  \bar X \rar{\bar \theta}\arrow{dr}{\hat \pi}& P = \IF_1\dar\\ && \IP^1  
   \end{tikzcd}
 \]
 with $\bar\theta(\bar E_1) = C_\infty$, the unique negative section. The branch divisor $\Delta$  of $\bar \theta$ is in $|4C_\infty + 8 \Gamma|$ with the following singularities:
 \begin{itemize}
  \item $\Delta = \Gamma_p + \Delta'$ has a $[3,3]$ point at $ \Gamma_p \cap C_\infty$ to produce the elliptic singularity of degree one;
  \item denoting the fourth point of $\Delta|_{C_\infty}$ with  $q\in C_\infty \setminus \Gamma_p$ and the fibre through $q$ with $\Gamma_q$ then $\Delta$ has a unique further singularity either of type $A_1$ (node) or $A_2$ (cusp) at a point of $\Gamma_q$, as given in Table \ref{tab: E13 sing fibre} for type $E_{13}$ and in Table \ref{tab: E14 sing fibre} for type $E_{14}$. 
 \end{itemize}
In other words, $W$ arises exactly as in  Example \ref{ex: E13} or Example \ref{ex: E14}.
\end{prooflist}
This concludes the proof of Theorem \ref{thm: type E}. \hfill\qed

\section{Surfaces with an exceptional unimodal double points of type  $Z_n$ or $W_n$}
\label{sect: Z and W}
We now turn to the exceptional unimodal double points of type $Z_n$ or $W_n$, which are listed in Table \ref{tab: info Z and W}. The analysis proceeds in a similar vain as in Section \ref{sect: En}, but, due to the nature of the singularities, the analysis is a bit simpler, because  the smooth birational models are special K3 surfaces already described in \cite{GPSZ22}.

In all these cases the exceptional divisor $E$ of $f$ is reduced and we have
\[ E^2 = - 2,\; K_X.E = 2, \text{ and } K_X^2 = -1.\]

\pgfmathwidth{"Kodaira"} 
\setlength{\Breite}{3.2cm}
\begin{table}
  \caption{The exceptional unimodal points of type $Z_n$ and $W_n$.}
  \label{tab: info Z and W}
 \begin{tabular}{ccccc}
 \toprule
 Type &\begin{minipage}[b]{\pgfmathresult pt} Kodaira\\ fibre \end{minipage} & Blow ups & $E$ & equation\\
 \midrule
 $Z_{11}$& $II$&$(2)$  &
 \begin{minipage}[c]{\Breite}
 \centering
 \begin{tikzpicture}[curve, scale = .5]
  \draw (1,2) to[out  = -90, in = 0] (0,0) to [in = 90,  out = 0] (1, -2);
\node at (1,-2.5) {$E_1$};
\node at (0.9,2.5){$-2$};
 \end{tikzpicture}
 \end{minipage}
 &
$yz^3 + y^5 +ay^4z$
\\
  $Z_{12}$&$III$&$(2,0)$ &
  \begin{minipage}[c]{\Breite}
 \centering
  \begin{tikzpicture}[curve, scale = .5]
\draw (1,2) to[out  = -135, in = 90] (0,0) to [in = 135,  out = -90] (1, -2);
\draw (-1,2) to[out  = -45, in = 90] (0,0) to [in = 45,  out = -90] (-1, -2);
\node at (-1,-2.5) {$E_1$};
\node at (1,-2.5) {$E_2$};
\node at (1,2.5){$-2$};
\node at (-1.2,2.5){$-4$};
 \end{tikzpicture}
 \end{minipage}
 &
 $yz^3 + y^4z + ay^3z^2$ 
 \\
  $Z_{13}$&$IV$&$(2,0,0)$  &
  \begin{minipage}[c]{\Breite}
 \centering
  \begin{tikzpicture}[curve, scale = .5]
\draw (0:2) to (180:2);
\draw (60:2.3) to (240:2.3);
\draw (120:2.3) to (300:2.3);
\node at (-1.2,-2.5) {$E_2$};
\node at (-2.6,0) {$E_1$};
\node at (1,2.5){$-2$};
\node at (-1.2,2.5){$-2$};
\node at (2.6,0) {$-4$};
\node at (1.2,-2.5) {$E_3$};
 \end{tikzpicture}
 \end{minipage}
 &
 $yz^3 + y^6 + ay^5z$ 
 \\
\midrule
  $W_{12}$&$III$&$(1,1)$  & \begin{minipage}[c]{\Breite}
 \centering
  \begin{tikzpicture}[curve, scale = .5]
\draw (1,2) to[out  = -135, in = 90] (0,0) to [in = 135,  out = -90] (1, -2);
\draw (-1,2) to[out  = -45, in = 90] (0,0) to [in = 45,  out = -90] (-1, -2);
\node at (-1,-2.5) {$E_1$};
\node at (1,-2.5) {$E_2$};
\node at (1,2.5){$-3$};
\node at (-1.2,2.5){$-3$};
 \end{tikzpicture}
 \end{minipage}
 &
$z^4+y^5+ay^3z^2$
 \\
  $W_{13}$&$IV$&$(1,1,0)$  & \begin{minipage}[c]{\Breite}
 \centering
  \begin{tikzpicture}[curve, scale = .5]
\draw (0:2) to (180:2);
\draw (60:2.3) to (240:2.3);
\draw (120:2.3) to (300:2.3);
\node at (-1.2,-2.5) {$E_2$};
\node at (-2.6,0) {$E_1$};
\node at (1,2.5){$-3$};
\node at (-1.2,2.5){$-2$};
\node at (2.6,0) {$-3$};
\node at (1.2,-2.5) {$E_3$};
 \end{tikzpicture}
 \end{minipage}
 &
$z^4 + y^4z + ay^6$
 \\
\bottomrule 
\end{tabular}
 \end{table}

\begin{thm}\label{thm: Z and W}
 Let $W$ be a Gorenstein surface with $K_W^2 = 1$, $p_g(W) = 2$ and $q(W)=0$ and a unique singular point which is of type $Z_{11}$ $Z_{12}$, $Z_{13}$,$W_{12}$, or $W_{13}$. Then there is a commutative diagram as in Figure \ref{fig: Z and W}. For each type we find one irreducible family, the general member of which is as described in \cite[Prop. 7.1]{GPSZ22}.
\end{thm}
 \begin{figure}[h]
  \caption{Construction of surfaces $W$ with one singularity of type $Z_n$ or $W_n$}
  \label{fig: Z and W}
 
 \begin{tikzcd}
X \dar{\phi} \arrow{rr}{f} && S\dar{\bar f}\\
W \arrow[<->, dashed]{rr}{\text{birational}} \dar{2:1}&& \bar S\dar{2:1}\\
\IP(1,1,2)\arrow[dashed]{dr} \arrow[<->, dashed]{rr}{\text{birational}} && \IP^2\arrow[dashed]{dl}{\text{ linear projection}}\\
& \IP^1
   \end{tikzcd}
 \begin{minipage}{.4\textwidth}
  \begin{itemize}
   \item $S$ a smooth K3 surface, 
   \item $\bar S$ a K3 surface with at most one canonical singularity,
   \item $f$ one blow up,
   \item $\bar f$, $\phi$ the minimal resolutions,
   \end{itemize}  
  \end{minipage}
 \end{figure}
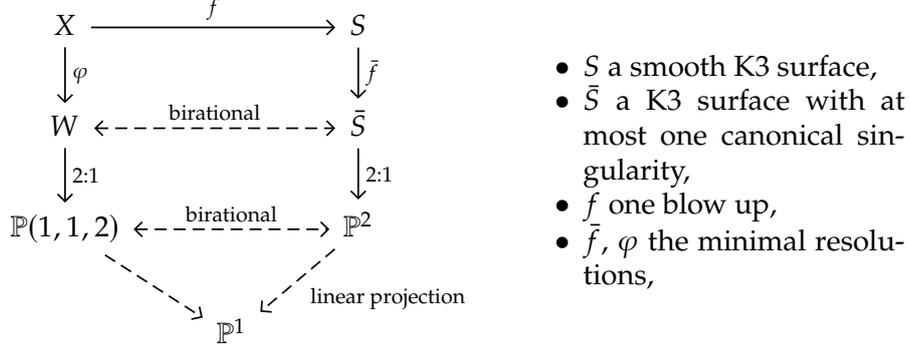

\subsection{Proof of Theorem \ref{thm: Z and W}}
\begin{lem}\label{K3} In the case of singularities of type $Z_{11}, Z_{12}, Z_{13},W_{12}$, and $W_{13}$, in Diagram \eqref{Diagram1} the surface $S$ is a K3 surface and $f\colon X \to S$ is one blow up.
\end{lem}
\begin{proof}

First, we note that $X$ cannot be a rational surface since $p_g(X)=1$. Thus, we have $\kappa(X)\geq 0$. We recall that $K_X^2=-1, K_X.E=2$ and $K_W^2=1$.

Assume for contradiction that  $X$ is of general type. By using that $K_X^2=-1$ and Lemma \ref{lem: basic stuff}, we obtain the inequality $-1 \leq K_S^2\leq 1$. Thus, we have that $K_S^2=1$, and then the minimal resolution $f\colon X\to S$ must be the composition of two blow ups, whose exceptional curves we call  $G_1$ and $G_2$. 

Hence,
\[f^*K_S+m_1G_1+m_2G_2 = K_X=\varphi^*K_W-E,\]

where $m_1,m_2$ are non negative integers. So, it follows that 
\[K_W^2=(\varphi^*K_W).f^*K_S.+(\varphi^*K_W).E+m_1(\varphi^*K_W).G_1+m_2(\varphi^*K_W).G_2,\]

and then
\[K_W^2=K_S^2+K_S.(f_*E)+m_1K_W.(\varphi_*G_1)+m_2K_W.(\varphi_*G_2),\]

where each term is non-negative. Since $K_W^2=K_S^2=1$ then we have that $K_S.(f_*E)=m_1K_W.(\varphi_*G_1)=m_2K_W.(\varphi_*G_1)=0$. It follows that $m_1\varphi_*G_1=m_2\varphi_*G_2=0$ and hence $S=X$. But then $0=K_S.(f_*E)=K_X.E=2$, which is a contradiction. Thus, we conclude that $\kappa(X)=0$ or $1$, and on the minimal model $S$ we have $K_S^2 = 0$. 

Since $K_X^2=-1$ by Lemma \ref{lem: basic stuff}, the map  $f\colon X \to S$ is one blow up.
We also have that $p_g(S)=p_g(X)=1$, so we can choose an effective canonical divisor $K_S$ on $S$. Then $K_X = f^* K_S +G$, where $G$ is the $(-1)$ curve in $X$. 

By Proposition \ref{prop: min res p_g=2}, $K_X = \tilde G + \sum a_i E_i $ where $\tilde G$ is reduced and irreducible. None of the $E_i$ satisfies $E_i^2 = -1$, so $\tilde G = G$ is the exceptional curve of the blow up and $\sum a_i E_i =f^* K_S$. If all $a_i =0$, then $K_S$ is trivial an $q=0$ implies that $S$ is a K3 surface.

So assume for contradiction that
$f^* K_S>0$, so $S$ is a relatively minimal elliptic fibration and $K_S$ is a union of multiples of  Kodaira fibres. 
But $-1 = K_X.\tilde G = \tilde G^2 + \sum a_i E_i.\tilde G$ shows that the blow down $f$ happens away from the support of $f^*K_S$, so $K_S$ is supported on a configuration of curves as given in Table \ref{tab: info Z and W}. This is impossible, so $S$ is a K3 surface as claimed.
\end{proof}

\begin{table}
 \caption{Curve configurations for the construction of Type $Z_{11}$ and $W_{12}$}
  \label{tab: without Sbar}
  \setlength{\Breite}{2.3cm}
 \begin{tabular}{ccccc}
 \toprule
  Type & in $X$  & in $S=\bar S$ 
  & branch in $\IP^2$ & \begin{minipage}{1.75cm}
  \small
      $\dim \IL(\Delta) -$ \\ $ \dim \Aut(\IP^2)$
  \end{minipage}
  \\   \midrule
\multirow{3}*{$Z_{11}$} & 
  \begin{minipage}[c]{\Breite}
 \centering \begin{tikzpicture}
\draw[E1] (0, 1) 
to[out=-80, in =80]  ++ (0,-1) to[out = 260, in = 0] ++(-.3, -.5)
to[out = 0 , in = 110] (0, -1)
 node[right]{$E_1$};
\draw[GG] (.1, .9) to[out = -150, in = 150] node[left]{$G$} ++ (0, -.8);
   \end{tikzpicture}  \end{minipage}
   &
  \begin{minipage}[c]{\Breite}
 \centering \begin{tikzpicture}
\draw[E1] 
(0, -1) node[right]{$\tilde E_1=\bar E_1$}
to[out = 110, in = 0] ++(-.3, .5)
to[out = 0 , in = 260] ++(.3, .5)
to[out = 80, in = 135] ++ (0.3,.5) to [out = -45, in = 35] ++ (-.5, -.5);

   \end{tikzpicture}  \end{minipage}
   
&
  \begin{minipage}[c]{\Breite}
 \centering \begin{tikzpicture}
\draw[E1] 
(0, -1) node[left] { $L$}
to
 ++(0,2);
\draw[branch] 
(0, .9) node[right] {$1$}
(.1, 1) 
to[out = 225, in = 90] ++ (-.2, -.4) 
to[out = -90, in = 90] ++ (.1, -.4) node[right] {$2$}
to[out = -90, in = 90] ++ (-.2, -.4) node[left] {$\Delta$}
to[out = -90, in = 90 ] ++ (.2, -.4) node[right] {$3$}
to[out = -90, in = 170] ++ (.4, -.4);
   \end{tikzpicture}  \end{minipage}
   
&$18$   
\\ 
 & 
  \begin{minipage}[c]{\Breite}
 \centering \begin{tikzpicture}
\draw[E1] (0, 1) 
to[out=-80, in =80]  ++ (0,-1) to[out = 260, in = 0] ++(-.3, -.5)
to[out = 0 , in = 110] (0, -1)
;
\draw[GG] (.27, .9)
to[out = -150, in = 150] 
++ (0, -.8);
   \end{tikzpicture}  \end{minipage}
   &
  \begin{minipage}[c]{\Breite}
 \centering \begin{tikzpicture}
\draw[E1] 
 (0, -1)
to[out = 110, in = 0] ++(-.3, .5)
to[out = 0 , in = 0] ++(0, 1)
to[out = 0, in = -115] ++ (0.3,.5);
   \end{tikzpicture}  \end{minipage}
  
&
  \begin{minipage}[c]{\Breite}
 \centering \begin{tikzpicture}
\draw[E1] (0, -1) to
++(0,2);
\draw[branch] 
(.2, 1)
to[out = 225, in = 90] ++ (-.2, -.5) node[right] {$3$}
to[out = -90, in = 90] ++ (-.4, -.5) 
to[out = -90, in = 90 ] ++ (.4, -.5) node[right] {$3$}
to[out = -90, in = 170] ++ (.4, -.5);
   \end{tikzpicture}  \end{minipage}
&$17$   
\\
& 
  \begin{minipage}[c]{\Breite}
 \centering \begin{tikzpicture}
\draw[E1] (0, 1) 
to[out=-80, in =0]  ++ (-.4,-1)
to[out = 0, in = 80] ++(.4, -1);
\draw[GG] (-.4, .5) 
to  ++ (0, -1);
   \end{tikzpicture}  \end{minipage}
   &
  \begin{minipage}[c]{\Breite}
 \centering \begin{tikzpicture}
\draw[E1] (0, 1) 
to[out=-60, in =0]  ++ (-.5,-1)
to[out = 0, in = 60] ++(.5, -1)
;
   \end{tikzpicture}  \end{minipage}
 &
  \begin{minipage}[c]{\Breite}
 \centering \begin{tikzpicture}
\draw[E1] (0, -1) to
++(0,2);
\draw[branch] 
(.4, -1)
to[out = 160, in = -90] ++ (-.4, .75) node[right] {$5$}
to[out = 90, in = -90] ++ (-.4, .75)
to[out = 90, in = 225 ] ++ (.6,.5)
(0, .82) node[right] {$1$};
   \end{tikzpicture}  \end{minipage}
&$17$   
\\ \midrule

$W_{12}$
&
  \begin{minipage}[c]{\Breite}
 \centering \begin{tikzpicture}
\draw[E1] (-.5, 1) to[out = -45, in =90] ++(.5,-1) to[out = -90, in =45] ++(-.5,-1) node[left] {$E_1$};
\draw[E1] (.5, 1) to[out = -135, in =90] ++(-.5,-1) to[out = -90, in =135] ++(.5,-1) node[right] {$E_2$};
\draw[GG] (-.5, .75) to ++ (1,0) node [below]{$G$};
   \end{tikzpicture}  
   \end{minipage}
&

  \begin{minipage}[c]{\Breite}
 \centering \begin{tikzpicture}
\draw[E1] (-.5, 1)  to[out = -80, in =150] ++ (.5, -.5) to[out = -30, in =90] ++(0,-1) to[out = -90, in =150] ++(.5,-.5);
\draw[E1] (.5, 1)  to[out = -110, in =30] ++ (-.5, -.5) to[out = -150, in =90] ++(0,-1) to[out = -90, in =30] ++(-.5,-.5);
   \end{tikzpicture}  
   \end{minipage}
 &
   \begin{minipage}[c]{\Breite}
 \centering \begin{tikzpicture}
\draw[E1] 
(0, -1) node[right] {$L$}
to
++(0,2);
\draw[branch] 
(-.3, 1)
to[out = -30, in = 90] ++ (.3, -.4) node[right] {$2$}
to[out = -90, in = 90] ++ (-.4, -.4) node[left] {$\Delta$}
to[out = -90, in = 90 ] ++ (.4, -.7) node[right] {$4$}
to[out = -90, in = 30] ++ (-.2, -.5);
   \end{tikzpicture} 
   \end{minipage}
&$17$   
   \\

&
  \begin{minipage}[c]{\Breite}
 \centering \begin{tikzpicture}
\draw[E1] (-.5, 1) to[out = -45, in =90] ++(.5,-1) to[out = -90, in =45] ++(-.5,-1);
\draw[E1] (.5, 1) to[out = -135, in =90] ++(-.5,-1) to[out = -90, in =135] ++(.5,-1);
\draw[GG] (-.5, 0) to ++ (1,0) 
;
   \end{tikzpicture}  
   \end{minipage}
  
 &  
  \begin{minipage}[c]{\Breite}
 \centering \begin{tikzpicture}
\draw[E1] (-.75, .75) to[out = -30, in =90] ++(.75,-.75) to[out = -90, in =30] ++(-.75,-.75);
\draw[E1] (.75, .75) to[out = -150, in =90] ++(-.75,-.75) to[out = -90, in =150] ++(.75,-.75);
   \end{tikzpicture}  
   \end{minipage}
   &
   \begin{minipage}[c]{\Breite}
 \centering \begin{tikzpicture}
\draw[E1] (0, -1) to
++(0,2);
\draw[branch] 
(-.3, 1)
to[in = 90, out = -60] ++ (.3, -1) node[right] {$6$}
to[out = -90, in = 60] ++ (-.3, -1);
   \end{tikzpicture} 
   \end{minipage}

&$16$     
   \\ \bottomrule
\end{tabular}
\end{table}

\begin{table}
 \caption{Curve configurations for the constructions of Type $Z_{12}$, $Z_{13}$ and $W_{13}$}
  \label{tab: with Sbar}
   \setlength{\Breite}{2.25cm}
 \begin{tabular}{cccccc}
 \toprule
  Type & in $X$  & in $S$ &  in $\bar S$ 
  & branch in $\IP^2$ & \begin{minipage}{2cm}
      $\dim \IL(\Delta) -$ \\ $ \dim \Aut(\IP^2)$
  \end{minipage}
  \\   \midrule

$Z_{12}$ & 
  \begin{minipage}[c]{\Breite}
 \centering \begin{tikzpicture}
\draw[GG] (-.157, 1) node[left]{$G$} to[out = -30, in = 30]  ++ (0, -.8);
\draw[E1] (0, 1) to ++(0, -2)
node[left] {$E_1$};
\draw[Ei] (.5, .6) to[out = -135, in =90] ++(-.5,-.8) to[out = -90, in =135] ++(.5,-.8) node[right] {$E_2$};
   \end{tikzpicture}  \end{minipage}
   &
  \begin{minipage}[c]{\Breite}
 \centering \begin{tikzpicture}
\draw[E1] (0, -1) node[left] {$\tilde E_1$}
-- ++(0,1.5)
to[out = 90, in = 135] ++ (0.3,.5)
to [out = -45, in = 35] ++ (-.5, -.5);
\draw[Ei] (.5, .6)
to[out = -135, in =90] ++(-.5,-.8) to[out = -90, in =135] ++(.5,-.8)node[right] {$\tilde E_2$} ;
   \end{tikzpicture}  \end{minipage}
&
  \begin{minipage}[c]{\Breite}
 \centering \begin{tikzpicture}
\draw[E1] 
(0, -1) node[right]{$\bar E_1$}
to[out = 110, in = 0] ++(-.3, .5)
to[out = 0 , in = 260] ++(.3, .5)
to[out = 80, in = 135] ++ (0.3,.5) to [out = -45, in = 35] ++ (-.5, -.5);
 \fill[sing] (-.3, -.5) node[above left]{$A_1$} circle (2pt);

   \end{tikzpicture}  \end{minipage}
&
  \begin{minipage}[c]{\Breite}
 \centering \begin{tikzpicture}
\draw[E1] 
(0, -1) 
node[right] { $L$}
to ++(0,2);

\draw[branch] 
(0, .9) node[right] {$1$}
(.1, 1)
to[out = 225, in = 90] ++ (-.2, -.4)
to[out = -90, in = 90] ++ (.1, -.4)  node[right] {$2$}
to[out = -90, in = 90] ++ (-.2, -.4) node[left] {$\Delta$}
to[out = -90, in = 90 ] ++ (.2, -.4)  node[above right] {$3$}
to[out = -90, in = 180, loop] ++ (0, 0) to[out = 0] ++ (.4, -.2);
   \end{tikzpicture}  \end{minipage}
   
&$17$
\\
& 
  \begin{minipage}[c]{\Breite}
 \centering \begin{tikzpicture}
\draw[GG] (-.21, 1)
to[out = -30, in = 30]  ++ (0, -.8);
\draw[E1] (0, 1) to ++(0, -2);
\draw[Ei] (.5, .6) to[out = -135, in =90] ++(-.5,-.8) to[out = -90, in =135] ++(.5,-.8);
   \end{tikzpicture}  \end{minipage}
   &
  \begin{minipage}[c]{\Breite}
 \centering \begin{tikzpicture}
\draw[E1] (0, -1) -- ++(0,1.75) to[out = -90, in = 10] ++ (-0.3,-.4);
\draw[Ei] (.5, .6) to[out = -135, in =90] ++(-.5,-.8) to[out = -90, in =135] ++(.5,-.8);
   \end{tikzpicture}  \end{minipage}
&
  \begin{minipage}[c]{\Breite}
 \centering \begin{tikzpicture}
\draw[E1] 
 (0, -1)
to[out = 110, in = 0] ++(-.3, .5)
to[out = 0 , in = 0] ++(0, 1)
to[out = 0, in = -115] ++ (0.3,.5);
 \fill[sing] (-.3, -.5) node[above left]{$A_1$} circle (2pt);
   \end{tikzpicture} 
   \end{minipage}
  &
    \begin{minipage}[c]{\Breite}
 \centering \begin{tikzpicture}
\draw[E1] (0, -1) to
++(0,2);
\draw[branch] 
(.2, 1)
to[out = 225, in = 90] ++ (-.2, -.5)  node[right] {$3$}
to[out = -90, in = 90] ++ (-.4, -.5)
to[out = -90, in = 90 ] ++ (.4, -.5)  node[above right] {$3$}
to[out = -90, in = 180, loop] ++ (0, 0) to[out = 0] ++ (.4, -.2);
   \end{tikzpicture}  \end{minipage}
&$16$
\\ \midrule
$Z_{13}$ 
 & 
  \begin{minipage}[c]{\Breite}
 \centering \begin{tikzpicture}
\draw[Ei] (-.4, 0)  to ++(.8,-.9) node[right] {$E_3$};
\draw[Ei] (-.4, -.9) 
node[left] {$E_2$}
to ++(.8,.9);

 \draw[GG] (-.157, 1) node[left]{$G$} to[out = -30, in = 30]  ++ (0, -.8);
\draw[E1] (0, 1)
node[right] {$E_1$}
to ++(0, -2);
   \end{tikzpicture}  \end{minipage}
   &
  \begin{minipage}[c]{\Breite}
 \centering \begin{tikzpicture}
\draw[E1] (0, -1) -- ++(0,1.5) to[out = 90, in = 135] ++ (0.3,.5) to [out = -45, in = 35] ++ (-.5, -.5) node[above] {$\tilde E_1$};
\draw[Ei] (-.4, 0) to ++(.8,-.9) node[right] {$\tilde E_3$};
\draw[Ei] (-.4, -.9) node[left] {$\tilde E_2$} to ++(.8,.9);
   \end{tikzpicture}  \end{minipage}
&
  \begin{minipage}[c]{\Breite}
 \centering \begin{tikzpicture}
\draw[E1] 
(0, -1) node[right]{$\bar E_1$}
to[out = 110, in = 0] ++(-.3, .5)
to[out = 0 , in = 260] ++(.3, .5)
to[out = 80, in = 135] ++ (0.3,.5) to [out = -45, in = 35] ++ (-.5, -.5);
 \fill[sing] (-.3, -.5) node[above left]{$A_2$} circle (2pt);

   \end{tikzpicture}  \end{minipage}
   &
  \begin{minipage}[c]{\Breite}
 \centering \begin{tikzpicture}
\draw[E1] (0, -1) node[right] {L}
to
++(0,2) ;
\draw[branch] 
(0, .9)  node[right] {$1$}
(.1, 1)
to[out = 225, in = 90] ++ (-.2, -.4) 
to[out = -90, in = 90] ++ (.1, -.4)  node[right] {$2$}
to[out = -90, in = 90] ++ (-.2, -.4) node[left] {$\Delta$}
to[out = -90, in = 90 ] ++ (.2, -.4)  node[right] {$3$}
to[out = 90] ++ (.4, .2);
   \end{tikzpicture}  \end{minipage}
  
&$16$
\\
& 
  \begin{minipage}[c]{\Breite}
 \centering \begin{tikzpicture}
\draw[Ei] (-.4, 0) to ++(.8,-.9);
\draw[Ei] (-.4, -.9) to ++(.8,.9);

 \draw[GG] (-.21, 1)
 to[out = -30, in = 30]  ++ (0, -.8);
\draw[E1] (0, 1) to ++(0, -2);
   \end{tikzpicture}  \end{minipage}
   &
  \begin{minipage}[c]{\Breite}
 \centering \begin{tikzpicture}
\draw[E1] (0, -1) -- ++(0,1.75) to[out = -90, in = 10] ++ (-0.3,-.4);
\draw[Ei] (-.4, 0) to ++(.8,-.9);
\draw[Ei] (-.4, -.9) to ++(.8,.9);

   \end{tikzpicture}  \end{minipage}
&
  \begin{minipage}[c]{\Breite}
 \centering \begin{tikzpicture}
\draw[E1] 
 (0, -1)
to[out = 110, in = 0] ++(-.3, .5)
to[out = 0 , in = 0] ++(0, 1)
to[out = 0, in = -115] ++ (0.3,.5);
 \fill[sing] (-.3, -.5) node[above left]{$A_2$} circle (2pt);
   \end{tikzpicture} 
   \end{minipage}
   &
       \begin{minipage}[c]{\Breite}
 \centering \begin{tikzpicture}
\draw[E1] (0, -1) to
++(0,2);
\draw[branch] 
(.2, 1)
to[out = 225, in = 90] ++ (-.2, -.5)  node[right] {$3$}
to[out = -90, in = 90] ++ (-.4, -.5)
to[out = -90, in = 90 ] ++ (.4, -.5) node[right] {$3$}
to[out = 90] ++ (.4, .2);

\end{tikzpicture}  \end{minipage}
&$15$

\\
\midrule
$W_{13}$
&
  \begin{minipage}[c]{\Breite}
 \centering 
 \begin{tikzpicture}[scale = .5]
\draw[E1] (0:2) to (180:2) node[above] {$E_1$};
\draw[E1] (60:2.3) to (240:2.3) node[left] {$E_2$};
\draw[Ei] (120:2.3) to (300:2.3) node[right] {$E_3$};
\draw[GG] ( 0, 1.2) to ++ (-60:2) node [right]{$G$};
   \end{tikzpicture}  
   \end{minipage}
&
  \begin{minipage}[c]{\Breite}
 \centering 
 \begin{tikzpicture}[scale = .5]
\draw[E1] (180:2) node[above] {$\tilde E_1$}
to (0,0) to[out = 0, in = -90]  (40:2);
\draw[E1] (240:2)node[left] {$ \tilde E_2$} to (0,0) to[out = 60, in = 150]  (20:2);
\draw[Ei] (120:2.3) to (300:2.3)  node[right] {$\tilde E_3$};
   \end{tikzpicture}  
   \end{minipage}
& 
\begin{minipage}[c]{\Breite}
 \centering \begin{tikzpicture}
\draw[E1] (-.5, 1)  to[out = -80, in =150] ++ (.5, -.5) to[out = -30, in =90] ++(0,-1) to[out = -90, in =150] ++(.5,-.5) node[right] {$\bar E_2$};
\draw[E1] (.5, 1)  to[out = -110, in =30] ++ (-.5, -.5) to[out = -150, in =90] ++(0,-1) to[out = -90, in =30] ++(-.5,-.5)node[left] {$\bar E_1$};
 \fill[sing] (0, -.5) node[above left]{$A_1$} circle (2pt);
   \end{tikzpicture}  
   \end{minipage}
 &
   \begin{minipage}[c]{\Breite}
 \centering \begin{tikzpicture}
\draw[E1] (0, -1) node[right] {L} to
++(0,2);
\draw[branch] 
(-.3, 1)
to[out = -30, in = 90] ++ (.3, -.4)  node[right] {$2$}
to[out = -90, in = 90] ++ (-.4, -.4)  node[left] {$\Delta$}
to[out = -90, in = 90 ] ++ (.4, -.7)  node[above right] {$3$}
to[out = -90, in = 0, loop] ++ (0,0) 
to[out = 180, in = 45] ++(-.3, -.2);
   \end{tikzpicture} 
   \end{minipage}
 &$16$  

\\
\bottomrule 
 \end{tabular}
\end{table}

\begin{lem}\label{double cover} In the case of singularities of type $Z_{11}, Z_{12}, Z_{13},W_{12}$, and $W_{13}$, the surface $X$ is birational to a  $K3$ surface with ADE singularities which is the double cover of $\IP^2$ branched along a plane sextic $\Delta$.
\end{lem}
\begin{proof}
By Lemma \ref{K3} the map $f\colon X\to S$ is one blow up. Let $G$ be the $(-1)$-curve contracted on $X$.
Recall that
\[G = f^*K_S+G= K_X=\varphi^*K_W-E.\]
Intersecting with $E$ we get $G.E = -E^2 = 2$. 
Therefore, after blowing down $G$, we obtain that $p_a(\tilde{E})=2$, and $\tilde{E}^2=2$, where $\tilde{E}=f_*E\subset S$ is the image of $E$. So, by Riemann-Roch we obtain that $h^0(\ko_S(\tilde{E}))=3$. In Table \ref{tab: without Sbar} and Table \ref{tab: with Sbar} we list the possible configurations of $\tilde{E}$.

Let us write $|\tilde{E}|=N+|M|$, where $N$ is the fixed part, and $|M|$ is the  mobile part of the linear system $|\tilde{E}|$. Since we have that $|M|$ does not have fixed part it is nef. Since  $S$ is a $K3$ surface, we obtain that $h^0(\ko_S(M))=h^0(\ko_S(\tilde{E}))=3$, and then $M^2=2$ and $M$ is big as well. By  \cite[Corollary 3.15, Chapter 2]{huybrechts2016lectures}, the fixed part is empty, the linear system  $|\tilde{E}|$ is base point free and the induced map  map $\psi\colon S\to \IP^2$  is a morphism of degree two.
The branch locus $\Delta$ is a plane sextic, since $K_{\bar S} = \bar\psi^*(K_{\IP^1}+1/2\Delta)$ is trivial.

In addition, analysing case by case in Table \ref{tab: without Sbar} and Table \ref{tab: with Sbar}, the sub-configuration formed by the $(-2)$-curves and $\tilde{E}$ intersect trivially. That is, all the $(-2)$-curves in $\tilde{E}$ are contracted by $\psi\colon S\to \IP^2$. So, we can factor the map $\psi=\bar{\psi}\circ\bar{f}$, where $\bar{f}\colon S\to \bar{S}$ is the map which contracts the $(-2)$-curves in $S$, and $\bar{\psi}\colon \bar{S}\to \IP^2$ is the double cover associated to the linear system $\bar{E}=\bar{f}_*(\tilde{E})$. The possible configurations in $\bar S$ with the singularities of the surfaces marked are also given in Table \ref{tab: without Sbar} and Table \ref{tab: with Sbar}. 
In particular, we have that $X$ is birational to the a $K3$ surface with ADE singularities.
\end{proof}
\begin{lem} A plane sextic $C$ is projectively equivalent to $\Delta$ if and only if
\begin{prooflist}
    \item[Type  $Z_{11}$] $C$ is smooth, and there exists a line $L$ such that either $C|_L=3p_1+2p_2+p_3$, $C|_L=3p_1+3p_2$, or $C|_L=5p_1+p_2$. 
    \item[Type  $W_{12}$] $C$ is smooth, and there exists a line $L$ such that either $C|_L=4p_1+2p_2$ or $C|_L=6p_1$.
    \item[Type  $W_{13}$] $C$ is smooth except for  an $A_1$ singularity at a point $p_1$, and there exists a line $L$ such that $C|_L=4p_1+2p_2$.
    \item[Type  $Z_{12}$] $C$ is smooth except for  an $A_1$ singularity at a point $p_1$, and there exists a line $L$ such that either $C|_L=3p_1+2p_2+p_3$ or $C|_L=3p_1+3p_2$.
    \item[Type  $Z_{13}$] $C$ is smooth except for  an $A_2$ singularity at a point $p_1$, and there exists a line $L$ such that either $C|_L=3p_1+2p_2+p_3$ or $C|_L=3p_1+3p_2$.
    
\end{prooflist} 
Moreover for any fixed type, let $\IL(\Delta)$ be the locus parametrising the plane sextics  with the given properties. Then $\IL(\Delta)$ is an open subset of a linear system, hence irreducible, and  $\dim (\IL(\Delta)) - \dim \Aut(\IP^2)$ is as in Table \ref{tab: without Sbar} and Table \ref{tab: with Sbar}.
\end{lem}

\begin{proof}
Let us first recall that $\tilde{E}$ is the image of $E$ after contracting the unique $(-1)$-curve $G$ on $X$, and $\bar{E}$ is its image after contracting the $(-2)$-curves on $\tilde{E}$, where $E$ is as in Table \ref{tab: info Z and W} for each case. Thus, we have $\tilde{E}=\bar{E}$ for the type $Z_{11}$, and $W_{12}$. In the following, we will analyze the geometry of $\bar{E}$ and $\Delta$ case by case.  As we saw in the proof of Lemma \ref{double cover}, we have that $E. G=2$, and so since $S$ is a $K3$ surface we have two consequences 
\begin{itemize}
    \item $G$ cannot intersect $(-2)$-curves on $E$. Otherwise, they would become $(-1)$-curves on $S$, but $S$ is minimal.

    \item $G$ must intersect the negative curves of $E$ with self-intersection less than $-2$. On the contrary, we would have rational curves with self-intersection less than $-2$ on a $K3$ surface, which is impossible.
\end{itemize}

Hence, we have that $G$ might intersect $E$ as in Table \ref{tab: without Sbar}, and Table \ref{tab: with Sbar}.
Let $[x:y:z]$ be the coordinates of $\IP^2$. We will denote by $f_5(x,y,z)$ a general polynomial of degree five in $\IP^2$.  Keeping the notation from Lemma \ref{double cover}, let $L$ be the image of $\bar{E}$ under the map $\bar{\psi}$.

As we saw in the proof of Lemma \ref{double cover}, we have that $p_a(\tilde E)=p_a(\bar{E})=2$, compare Table \ref{tab: without Sbar} and Table \ref{tab: with Sbar}. Thus, $L=\bar\varphi(\bar E)$ cannot be contained  in $\Delta$ because there are no curves of arithmetic genus two in $\IP^2$. Now, after a change of coordinates, we can assume that  $L=\{z=0\}$. 

\begin{description}[labelindent= 0cm, leftmargin=\parindent] 
\item[Type $Z_{11}$:] We treat the cases from Table \ref{tab: without Sbar}.
\begin{prooflist}[style = nextline, labelindent= 0cm,  leftmargin=\parindent]
\item[Case 1:  The curve $G$ does not pass through the cusp point of $E$, and intersects two different points of $E$.] So, the curve $\bar{E}$ has singularity of type $A_2$ (cusp) and $A_1$ (node)(see Table \ref{tab: without Sbar}). Thus, $\Delta|_L$ consists of three points $p_1,p_2,p_3$ with multiplicity three, two and one respectively. By making a change of coordinates, we may assume that $p_1=[1:0:0]$, and that $p_3=[1:1:0]$. Thus, we have $\Delta|_{L}=3[1:0:0]+2[1:\lambda:0]+[1:1:0]$, where $\lambda\neq 0,1$. So, we can write
\begin{equation*}
    \Delta=\{y^3(\lambda x-y)^2(x-y)+zf_5(x,y,z)=0\}.
\end{equation*}
Any automorphism of $\IP^2$ can be represented by an invertible matrix
\[ \mat{ a_{11} & a_{12}& a_{13}\\ a_{21} & a_{22}&a_{23}\\ a_{31}& a_{32} &a_{33}},\]

and so, the automorphisms that fix the points $[1:0:0]$ and $[1:1:0]$ on $\IP^2$ are represented by an invertible matrix of the form

\[ \mat{ 1 & 0& a\\ 0 & b&c\\ 0& 0 &d},\] 
where $b$ and $d$ are nonzero. Conversely, every matrix of that form with $b$ and $d$ nonzero gives an automorphism of $\IP^2$ fixing the two points. Thus, we have that the dimension of the stabiliser of  the two points is $4$. Putting this together with the facts that $\lambda$ is a parameter, and $h^0(\IP^2,\ko(5))=21$, we compute that $\dim (\IL(\Delta)) - \dim \Aut(\IP^2)=1+21-4=18$.
\item[Case 2:  The curve $G$ does not pass through the cusp point of $E$, and is tangent to $E$.]
So, the curve $\Bar{E}$ has two singularities of type $A_2$ (cusp) (see Table \ref{tab: without Sbar}). Thus, $\Delta|_L$ consists of two points $p_1,p_2$ with multiplicity three. By making a change of coordinates, we may assume that $p_1=[1:0:0]$, and that $p_2=[1:1:0]$. That is, we have $\Delta|_{L}=3[1:0:0]+3[1:1:0]$. Then,
\begin{equation*}
    \Delta=\{y^3(x-y)^3+zf_5(x,y,z)=0\}.
\end{equation*}

\item[Case 3: The curve $G$ intersects $E$ at the cusp point $p$.]
Since $E.G=2$ their tangent cones do not intersect, that is, they are separated after one blow up. Then we obtain that $\Bar{E}$ has a unique singularity of type $A_{4}$. In this case, we have that $\Delta|_L$ consists of two points $p_1,p_2$ with multiplicity five and one respectively. By making a change of coordinates, we may assume that $p_1=[1:0:0]$, and that and that the sixth point in the intersection is $[1:1:0]$. Thus, we have $\Delta|_{L}=5[1:0:0]+[1:1:0]$. Then,
\begin{equation*}
    \Delta=\{y^5(x-y)+zf_5(x,y,z)=0\}.
\end{equation*}
\end{prooflist}
Due to the choice of coordinates, we observe that the families in the last two cases are obtained by taking $\lambda =1$, and $\lambda=0$ in the first family, respectively. Similar to Case 1, using that $\lambda$ is fixed, we compute that $\dim (\IL(\Delta)) - \dim \Aut(\IP^2)=21-4=17$ in Case 2 and Case 3.

 Observe that  $S=\bar{S}$ is smooth and so the branch $\Delta$ is smooth.

\item[Type $W_{12}$:] We treat the cases from Table \ref{tab: without Sbar}.
\begin{prooflist}[style = nextline, labelindent= 0cm,  leftmargin=\parindent]
    \item[Case 1: The curve $G$ intersects once $E_1$, and once $E_2$.]
That is, the curve $\bar{E}$ has a singularity of type $A_3$ and $A_1$. Thus, $\Delta|_L$ consists of two points $p_1,p_2$ with multiplicity four, and two respectively. By making a change of coordinates, we may assume that $p_1=[1:0:0]$. Thus, we have $\Delta|_{L}=4[1:0:0]+2[1:\lambda:0]$, where $\lambda\neq 0$. So, we can write
\begin{equation*}
    \Delta=\{y^4(y-\lambda x)^2+zf_5(x,y,z)=0\}.
\end{equation*}
By using the same argument as in the previous cases, we compute that $\dim (\IL(\Delta)) - \dim \Aut(\IP^2)=1+21-5=17$.
    
      \item[Case 2: The curve $G$ intersects the singularity of type $A_3$ in $E$.]  That is, the curve $\Bar{E}$ has a singularity of type $A_5$ and $\Delta|_L$ consists in one point $p_1$ with multiplicity six. By making a change of coordinates, we may assume that $p_1=[1:0:0]$. That is, we have $\Delta|_{L}=6[1:0:0]$, and so
     \begin{equation*}
    \Delta=\{y^6+zf_5(x,y,z)=0\}.
\end{equation*}
Again, we have that the family in the second case is obtained by making $\lambda=0$ in the first family. And, then we compute that $\dim (\IL(\Delta)) - \dim \Aut(\IP^2)=21-5=16$. 
\end{prooflist}
Observe that  $S=\bar{S}$ is smooth and so the branch $\Delta$ is smooth.

\item[Type $W_{13}$:] We treat the only case from Table \ref{tab: with Sbar}.
\begin{prooflist}[style = nextline, labelindent= 0cm,  leftmargin=\parindent]
    \item[Case 1: The curve $G$ intersects once the components $E_1$, and $E_2$.]
    That is, the curve $\tilde{E}\subset S$ has a singularity of type $A_3$ and $A_1$. After contracting the $(-2)$-curve on the configuration of $\tilde{E}$ we obtain that $\bar{S}$ has a surface singularity of type $A_1$. The curve $\bar{E}\subset \bar{S}$ has a singularity of type $A_3$ at the $A_1$-point in $\bar{S}$, and a singularity of type $A_1$ (nodal curve singularity). Thus, $\Delta$ has an $A_1$ singularity (surface singularity) at a point $p_1$, and $\Delta_L$ is the union of $p_1$ with multiplicity four, and another double point $p_2$.By making a change of coordinates, we may assume that $p_1=[1:0:0]$, and that $p_2=[0:1:0]$.
    So, we can write
\begin{equation*}
    \Delta=\{y^4x^2+zf_5(x,y,z)=0\}.
\end{equation*}
Now, in order to impose an $A_1$ singularity over the point $[1:0:0]$, we need $f_5(1,0,0)=0$. That implies that $f_5(x,y,z)$ cannot contain the term $x^5$. Note that since the general member on $\Delta$ has the terms $z^2x^4$, and $zyx^4$, then the general member  has the desired conditions. Moreover, since we fixed two points on $\IP^2$, and $f_5(x,y,z)$ cannot contain the term $x^5$, we compute that $\dim (\IL(\Delta)) - \dim \Aut(\IP^2)=21-1-4=16$.
\end{prooflist}
\item[Type $Z_{12}$:] We treat the cases from Table \ref{tab: with Sbar}.
\begin{prooflist}[style = nextline, labelindent= 0cm,  leftmargin=\parindent]
 \item[Case 1: The curve $G$ intersects two different points of $E_1$.]
 So, the curve $\tilde{E}\subset S$ has a singularity of type $A_3$, and a singularity of type $A_1$. After contracting the $(-2)$-curve on the configuration of $\tilde{E}$, we obtain that $\bar{S}$ has a surface singularity of type $A_1$. The curve $\bar{E}\subset \bar{S}$ has a singularity of type $A_3$ at the $A_1$-point in $\bar{S}$, and a singularity of type $A_1$ (curve singularity). Thus, $\Delta$ has an $A_1$ singularity (surface singularity) at a point $p_1$, and $\Delta|_L$ is the union of $p_1$ with multiplicity three, a double point $p_2$, and a point $p_3$ with multiplicity one. By making a change of coordinates, we can assume that $p_1=[1:0:0]$, and that $p_2=[0:1:0]$. Thus, we have $\Delta|_{L}=3[1:0:0]+2[0:1:0]+[\lambda:1:0]$, where $\lambda\neq 0$. So, we can write
\begin{equation*}
    \Delta=\{y^3x^2(x-\lambda y)+zf_5(x,y,z)=0\}.
\end{equation*}
In order to impose an $A_1$ singularity over the point $[1:0:0]$, we need $f_5(1,0,0)=0$. That implies that $f_5(x,y,z)$ cannot contain the term $x^5$. Note that since the general member on $\Delta$ has the terms $z^2x^4$, and $zyx^4$, then the general member  has the desired conditions. Since we fixed two points on $\IP^2$ and $\lambda$ is a parameter, we compute that $\dim (\IL(\Delta)) - \dim \Aut(\IP^2)=1+21-1-4=17$.

\item[Case 2: The curve $G$ is tangent to $E_1$.]
So, the curve $\tilde{E}\subset S$ has a singularity of type $A_3$, and $A_1$. After contracting the $(-2)$-curve on the configuration of $\tilde{E}$, we obtain that $\bar{S}$ has a surface singularity of type $A_1$. The curve $\bar{E}\subset \bar{S}$ has a singularity of type $A_3$ at the $A_1$-point in $\bar{S}$, and $A_3$ (curve singularity). Thus, $\Delta$ has an $A_1$ singularity (surface singularity) at a point $p_1$, and $\Delta|_L$ is the union of $p_1$ with multiplicity three, a point $p_2$ with multiplicity three. By making a change of coordinates, we may assume that $p_1=[1:0:0]$, and that $p_2=[0:1:0]$. That is, we have $\Delta|_{L}=3[1:0:0]+3[0:1:0]$. Then,
\begin{equation*}
    \Delta=\{y^3x^3+zf_5(x,y,z)=0\},
\end{equation*}
where $f_5(x,y,z)$ cannot contain the term $x^5$. Note that since the general member on $\Delta$ has the terms $z^2x^4$, and $zyx^4$, then the general member  has the desired conditions. We compute that $\dim (\IL(\Delta)) - \dim \Aut(\IP^2)=21-1-4=16$.
\end{prooflist}

\item[Type $Z_{13}$:] We treat the cases from Table \ref{tab: with Sbar}.
\begin{prooflist}[style = nextline, labelindent= 0cm,  leftmargin=\parindent]
    \item[Case 1: The curve $G$ intersects two different points of $E_1$.]
    So, the curve $\tilde{E}\subset S$ has a singularity of type $A_3$, and a singularity of type $A_1$. After contracting the $(-2)$-curves on the configuration of $\tilde{E}$, we obtain that $\bar{S}$ has a surface singularity of type $A_2$. The curve $\bar{E}\subset \bar{S}$ has a singularity of type $A_3$ at the $A_2$-point in $\bar{S}$, and $A_1$ (curve singularity). Thus, $\Delta$ has an $A_2$ singularity (surface singularity) at a point $p_1$, and $\Delta_L$ is the union of $p_1$ with multiplicity three, a double point $p_2$, and a point $p_3$ with multiplicity one. By making a change of coordinates, we may assume that $p_1=[1:0:0]$, and that $p_2=[0:1:0]$. Thus, we have $\Delta|_{L}=3[1:0:0]+2[0:1:0]+[\lambda:1:0]$ where $\lambda\neq 0$. So, we can write
\begin{equation*}
    \Delta=\{y^3x^2(x-\lambda y)+zf_5(x,y,z)=0\}.
\end{equation*}
In order to impose an $A_2$ singularity over the point $[1:0:0]$, we need $f_5(1,0,0)=0$. That implies that $f_5(x,y,z)$ cannot contain the term $x^5$. In this case, if both terms $zx^4$ and $yx^4$ have coefficients not zero at the same time in  $f_5(x,y,z)$, then we have that the general member on $\Delta$ would have an $A_1$ singularity as in the previous case. Hence, we can assume that $f_5$ does not contain the term $yx^4$, and so the general member on $\Delta$ has the terms $zyx^4$, and $z^3x^3$, then the general member  has the desired conditions. Since we fixed two points on $\IP^2$, $\lambda$ is a parameter, and $f_5(x,y,z)$ cannot contain the terms $x^5$ or $yx^4$, we compute that $\dim (\IL(\Delta)) - \dim \Aut(\IP^2)=1+21-2-4=16$.

\item[Case 2: The curve $G$ is tangent to $E_1$.]
So, the curve $\tilde{E}\subset E$ has a singularity of type $A_3$, and $A_1$. After contracting the $(-2)$-curves on the configuration of $\tilde{E}$, we obtain that $\bar{S}$ has a surface singularity of type $A_2$. The curve $\bar{E}\subset \bar{S}$ has a singularity of type $A_3$ at the $A_2$-point in $\bar{S}$, and $A_3$ (curve singularity). Thus, $\Delta$ has an $A_2$ singularity (surface singularity) at a point $p_1$, and $\Delta_L$ is the union of $p_1$ with multiplicity three, a point $p_2$ with multiplicity three. By making a change of coordinates, we may assume that $p_1=[1:0:0]$, and that $p_2=[0:1:0]$. That is, we have $\Delta|_{L}=3[1:0:0]+3[0:1:0]$. Then,
\begin{equation*}
    \Delta=\{y^3x^3+zf_5(x,y,z)=0\},
\end{equation*}
where $f_5(x,y,z)$ cannot contain the term $x^5$, and so we have that $\dim (\IL(\Delta)) - \dim \Aut(\IP^2)=21-2-4=15$.
\end{prooflist}
\end{description}
We have treated all cases and conclude the proof.\end{proof}


\section{Application to the moduli space of stable surfaces $\overline\gothM_{1,3}$}
\label{sect: application to moduli}
Let us spell out the consequences for the study of the stable compactification $\overline\gothM_{1,3}$,  the moduli space of stable I-surfaces.

Let us denote by 
\begin{equation}\label{eq: eight divisors} \kd(E_{11}),\, \kd(E_{12}),\, \kd(E_{13}),\, \kd(Z_{11}),\, \kd(Z_{12}),\, \kd(Z_{13}),\, \kd(W_{11}),\, \kd(W_{12})
 \end{equation}
the (closure of the) eight distinct divisors in the closure of the classical component constructed in \cite{GPSZ22}.
Recall that in their construction, an explicit equation was chosen for a branch divisor that contains one of the unimodal singularities, from which the stable replacement is constructed explicitly. In particular, they do not exclude the existence of other families, where the unimodal point on the branch divisor arises in a different way, compare \cite[Rem. 4.5]{GPSZ22}. 

\begin{thm}
 Let $\pi \colon \kx \to B$ be a flat family of surfaces such that for a point $0\in B$ the family over $B\setminus \{0\}$ is an admissible\footnote{For the moduli of stable surfaces, families have to satisfy an additional assumption beyond flatness, often called $\IQ$-Gorenstein, compare \cite{KollarModuli}. In our case, flatness is enough,  because the central fibre is Gorenstein and thus the family is Gorenstein close to the central fibre.}
 family of I-surfaces and $\kx_0:= W$ is a surface with a unique singular point, which is an exceptional unimodal double point. Then
 \begin{enumerate}
  \item The stable replacement $W^\text{st}$ of $W$ lies in one of the divisors in \eqref{eq: eight divisors}.
  \item If the singular point is of type $E_n$, then $W^\text{st}$ is birationally  the union of a minimally elliptic surface as described in Section \ref{sect: En}  and a K3 surface.
  \item If the singular point is of type $Z_n$ or $W_n$, then $W^\text{st}$ is birationally  the union of two K3 surfaces as described in \cite{GPSZ22}. 
 \end{enumerate}
\end{thm}
\begin{proof}
 We have proved in Theorem \ref{thm: type E} and Theorem \ref{thm: Z and W} that there is one irreducible family of such  $W$, which therefore has to coincide with the one constructed in \cite{GPSZ22}. Therefore the stable replacement of $W$ lies in one of the divisors constructed in loc.\ cit.\ . For $(ii)$ and $(iii)$ we note that again by  \cite{GPSZ22} the stable replacement is the union of a weighted blow up of $W$ and a singular K3 surface, so the only missing information is the geometry of the minimal resolution of $W$ in the $E_n$ case, which is provided by Theorem \ref{thm: type E}.
\end{proof}
It would be interesting to understand how these divisors intersect and how they interact with the divisors coming from T-singularities described in \cite{CFPRR}, but we do not address this question here.

\begin{rem}
 A similar analysis could in principle be done for Gorenstein surfaces with $K_X^2 = 1$ and $p_g = 1$ thus complementing the results in \cite{FPR17} and \cite{do-rollenske22}. 
 However, as happened with the simple elliptic singularities considered in \cite{do-rollenske22}, one should expect more cases depending on whether the canonical curve passes through the singularity or not. 
\end{rem}

\bibliographystyle{alpha}
\bibliography{references}

\end{document}